\documentclass[11pt]{amsart}

\usepackage{mathpple}
\usepackage{amsmath,amsfonts}
\usepackage{amsrefs}
\usepackage{amsthm}
\usepackage[all]{xy}

\usepackage{hyperref}

\numberwithin{equation}{section}

\newcommand{\op}{\operatorname}
\newcommand{\C}{\mathbb{C}}

\newcommand{\R}{\mathbb{R}}
\newcommand{\Q}{\mathbb{Q}}
\newcommand{\Z}{\mathbb{Z}}
\newcommand{\Etau}{{\text{E}_\tau}}
\newcommand{\E}{{\mathcal E}}
\newcommand{\F}{\mathbf{F}}
\newcommand{\G}{\mathbf{G}}
\newcommand{\eps}{\epsilon}
\newcommand{\g}{\mathbf{g}}
\newcommand{\im}{\op{im}}
\newcommand{\abracket}[1]{\left\langle#1\right\rangle}
\newcommand{\bbracket}[1]{\left[#1\right]}
\newcommand{\fbracket}[1]{\left\{#1\right\}}
\newcommand{\bracket}[1]{\left(#1\right)}

\newcommand{\mc}{\mathcal}
\newcommand{\cinfty}{C^{\infty}}
\newcommand{\pa}{\partial}

\renewcommand{\dbar}{\bar\pa}
\newcommand{\OO}{{\mathcal O}}
\newcommand{\hotimes}{\hat\otimes}
\newcommand{\BV}{Batalin-Vilkovisky }
\newcommand{\CE}{Chevalley-Eilenberg }
\newcommand{\suml}{\sum\limits}
\newcommand{\prodl}{\prod\limits}
\newcommand{\into}{\hookrightarrow}
\newcommand{\Ol}{\mathcal O_{loc}}

\newcommand{\iso}{\cong}
\newcommand{\dpa}[1]{{\pa\over \pa #1}}
\newcommand{\PP}{\mathrm{P}}
\newcommand{\Kahler}{K\"{a}hler }

\renewcommand{\Im}{\op{Im}}
\renewcommand{\Re}{\op{Re}}
\DeclareMathOperator{\Sym}{Sym}
\DeclareMathOperator{\Hom}{Hom}

\DeclareMathOperator{\Tr}{Tr}

\DeclareMathOperator{\PV}{PV}
\DeclareMathOperator{\Der}{Der}
\DeclareMathOperator{\HW}{HW}
\DeclareMathOperator{\Eu}{Eu}
\DeclareMathOperator{\HH}{H}

\theoremstyle{plain}
\newtheorem{thm}{Theorem}[section]
\newtheorem{thm-defn}{Theorem/Definition}[section]
\newtheorem{lem}{Lemma}[section]
\newtheorem{lem-defn}{Lemma/Definition}[section]
\newtheorem{prop}{Proposition}[section]
\newtheorem{cor}{Corollary}[section]

\theoremstyle{definition}
\newtheorem{defn}{Definition}[section]

\theoremstyle{remark}
\newtheorem{rmk}{Remark}[section]

\allowdisplaybreaks[4]  %%%%%%%%%%% choose 1-4, where 4 is the strongest desire to breack

\begin{document}

 \title{{BCOV theory on the elliptic curve and higher genus mirror symmetry}}
  \author{Si Li}
  \date{}

  % 封面
  \maketitle

%%%%%%%%%%%%%%%%%%%%%%%%%%%%%%
%% 前言部分
%%%%%%%%%%%%%%%%%%%%%%%%%%%%%%

\begin{abstract} We develop the quantum Kodaira-Spencer theory on the elliptic curve and establish the corresponding higher genus B-model. We show that the partition functions of the higher genus B-model on the elliptic curve are almost holomorphic modular forms, which can be identified with partition functions of descendant Gromov-Witten invariants on the mirror elliptic curve. This gives the first compact Calabi-Yau example where mirror symmetry is established at all genera.
\end{abstract}

  % 目录
  \tableofcontents
  % 表格目录
%  \listoftables
  % 插图目录
%  \listoffigures

%%%%%%%%%%%%%%%%%%%%%%%%%%%%%%
%% 正文部分
%%%%%%%%%%%%%%%%%%%%%%%%%%%%%%
%\mainmatter

\section{Introduction}
Mirror symmetry originated from string theory as a duality between
superconformal field theories (SCFT). The natural geometric
background involved is the Calabi-Yau manifold, and SCFTs can
be realized from twisting $\sigma$-models on Calabi-Yau manifolds in
two different ways \cite{Witten-Sigma-Model, Witten-Mirror}: the A-model and the B-model. The physics
statement of mirror symmetry says that the A-model on a Calabi-Yau
manifold $X$ is equivalent to the B-model on a different Calabi-Yau
manifold $\breve{X}$, which is called the \emph{mirror}.

The mathematical interests on mirror symmetry started from the work
\cite{Candelas-mirror}, where a remarkable mathematical prediction
was extracted from the physics statement of mirror symmetry:
the counting of rational curves on the Quintic 3-fold could be computed from the period integrals on the mirror Quintic 3-fold. Motivated by
this example, it's believed that such phenomenon holds for
general mirror Calabi-Yau manifolds. The counting of rational curves is
refered to as the genus $0$ A-model, which has now been
mathematically established \cite{Ruan-Tian, Li-Tian} as \emph{Gromov-Witten theory}. The
period integral is related to the \emph{variation of Hodge
structure}, and this is refered to as the genus $0$ B-model. Mirror
conjecture at genus $0$ has been proved by Givental
\cite{Givental-mirror} and Lian-Liu-Yau \cite{LLY} for a large class
of Calabi-Yau manifolds inside toric varieties.

A fundamental mathematical question is to understand mirror symmetry at higher genus. In the A-model,
the Gromov-Witten theory has been established for curves of arbitrary genus, and the problem of  counting higher genus  curves on Calabi-Yau manifolds has a solid mathematical foundation. However, the higher genus B-model is much more mysterious. One mathematical approach to the higher genus B-model by Costello \cite{Costello-partition} is categorical, from the viewpoint of Kontsevich's homological mirror symmetry \cite{Kontsevich-HMS}. The B-model partition function is proposed through the Calabi-Yau A-infinity category of coherent sheaves and a classification of certain 2-dimensional topological field theories. Unfortunately, the computation from categorical aspects is extremely difficult.

In the physics literature, Bershadsky-Cecotti-Ooguri-Vafa proposed a closed string field theory interpretation of the B-model, and suggested that the B-model partition function could be constructed from a quantum field theory, which is called the Kodaira-Spencer gauge theory of gravity  \cite{BCOV}. The solution space of the classical equations of motion in Kodaira-Spencer gauge theory describes the deformation space of complex structures on the underlying Calabi-Yau manifold, from which we can recover the well-known geometry of the genus 0 B-model.  We will call this quantum field theory as BCOV theory. Based on such idea, a non-trivial prediction has been made in \cite{BCOV}, which says that the genus one partition function  in the B-model on Calabi-Yau 3-fold is given by certain holomorphic Ray-Singer torsion and it could be identified with the genus one Gromov-Witten invariants on the mirror Calabi-Yau 3-fold. This is recently partially confirmed by Zinger  \cite{Zinger-BCOV}.

The original BCOV theory was defined only for Calabi-Yau 3-folds. In our previous work \cite{Si-Kevin}, a variant of the classical BCOV theory has been proposed which works for Calabi-Yau manifolds of any dimension and also contains the gravitational descendants. We initiated in \cite{Si-Kevin} a mathematical analysis of the quantum geometry of perturbative BCOV theory based on the effective renormalization method developed in \cite{Costello-book}, and explored various relations between the quantization and the higher genus B-model. As an example, a canonical quantization of BCOV theory has been constructed on elliptic curves in \cite{Si-Kevin,open-closed-BCOV} which satisfies string equation and dilaton equation mirror to the Gromov-Witten theory.

In the current paper, we will focus on the quantum BCOV theory on the elliptic curve and establish its equivalence with Gromov-Witten theory on the mirror.
We will give a brief description of the main results in this introduction.

\subsection{The A-model and Gromov-Witten theory}\label{Intro-A-model}
Let $X$ be a smooth projective algebraic variety with complexified \Kahler form $\omega_X$, where $\Re \omega_X$ is a \Kahler form, and $\op{Im} \omega_X\in \HH^2(X, \R)/\HH^2(X, \Z)$. The Gromov-Witten theory on $X$ concerns the moduli space
$$
  \overline{M}_{g,n,\beta}(X)
$$
parametrizing Kontsevich's stable maps \cite{Kontsevich-Torus} $f$ from connected, genus $g$, nodal curve $C$ to $X$, with $n$ distinct smooth marked points, such that
$$
     f_*[C]=\beta\in \HH_2(X, \Z)
$$
This moduli space is equipped with evaluation maps
\begin{eqnarray*}
   ev_i: \overline{M}_{g,n,\beta}(X)&\to& X\\
      \bbracket{f, \bracket{C;p_1,\cdots, p_n}} &\to&  ev_i\bracket{\bbracket{f, \bracket{C;p_1,\cdots, p_n}} }=f(p_i)
\end{eqnarray*}
The cotangent line to the $i^{th}$ marked point is a line bundle on $\overline{M}_{g,n,\beta}(X)$, whose first Chern class will be denoted by $\psi_i\in \HH^2\bracket{\overline{M}_{g,n,\beta}(X)}$. The Gromov-Witten invariants of $X$ are defined by
\begin{eqnarray*}
         \abracket{-}: \Sym^n_{\C}\bracket{\HH^*(X)[[t]]}&\to& \C\\
             \abracket{t^{k_1}\alpha_1,\cdots, t^{k_n}\alpha_n}_{g,n,\beta}^X&=&\int_{\bbracket{\overline{M}_{g,n,\beta}(X)}^{vir}} \psi_1^{k_1}ev_1^*\alpha_1\cdots \psi_n^{k_n}ev_n^*\alpha_n
\end{eqnarray*}
where $\bbracket{\overline{M}_{g,n,\beta}(X)}^{vir}$ is the virtual fundamental class \cite{Li-Tian, Intrinsic-normal-cone} of $\overline{M}_{g,n,\beta}(X)$, which is a homology class of dimension
\begin{eqnarray}\label{moduli-dim}
        \bracket{3-\dim X}(2g-2)+2\int_\beta c_1(X)+2n
\end{eqnarray}

In the Calabi-Yau case, the dimension of the virtual fundamental class doesn't depend on $\beta$ since $c_1(X)=0$.

\begin{defn}\label{A-model correlation}
The genus $g$ A-model partition function $\F^A_{g,n,X;q}[-]$ with $n$ inputs is defined to be the multi-linear map
\begin{eqnarray*}
          \F^A_{g,n,X;q}: \Sym^n_{\C}\bracket{\HH^*(X, \C)[[t]]}&\to& \C\\
              \F^A_{g,n,X;q}\bbracket{t^{k_1}\alpha_1,\cdots, t^{k_n}\alpha_n}&=&\sum_{\beta\in \HH_2(X, \Z)}q^{\int_\beta \omega_X}\abracket{t^{k_1}\alpha_1,\cdots, t^{k_n}\alpha_n}_{g,n,\beta}^X
\end{eqnarray*}
where  $q$ is a formal variable.
\end{defn}
The A-model partition function satisfies the following basic properties
\begin{enumerate}
\item Degree Axiom. $ \F^A_{g,n,X;q}\bbracket{t^{k_1}\alpha_1,\cdots, t^{k_n}\alpha_n}$ is non-zero only for
$$
      \suml_{i=1}^n \bracket{\deg \alpha_i+2k_i}=(2g-2)\bracket{3-\dim X}+2n
$$
Moreover, we have the Hodge decomposition $\HH^n(X, \C)=\bigoplus\limits_{p+q=n}\HH^{p,q}$. If we define the Hodge weight of $t^k\alpha\in t^k\HH^{p,q}$ by $\HW(t^k\alpha)=k+p-1$, then the reality condition implies  the Hodge weight condition
$$
   \suml_{i=1}^n \HW(\alpha_i)=(g-1)\bracket{3-\dim X}
$$

\item String equation. $\F^A_{g,n,X;q}$ satisfies the string equation
$$
     \F^A_{g,n+1,X;q}\bbracket{1,t^{k_1}\alpha_1,\cdots,t^{k_n}\alpha_n}=\suml_{i=1}^n\F^A_{g,n,X;q}\bbracket{t^{k_1}\alpha_1,\cdots,t^{k_i-1}\alpha_i,\cdots t^{k_n}\alpha_n}
     $$
for $2g-2+n>0$.
\item Dilaton equation. $\F^A_{g,n,X;q}$ satisfies the dilaton equation
$$
         \F^A_{g,n,X;q}\bbracket{t,t^{k_1}\alpha_1,\cdots,t^{k_n}\alpha_n}=(2g-2+n)\F^A_{g,n,X;q}\bbracket{t^{k_1}\alpha_1,\cdots,t^{k_n}\alpha_n}
$$
\end{enumerate}
The parameter $q$ can be viewed as the \Kahler moduli. Since the Gromov-Witten invariants are invariant under complex deformations, $\F^A_{g,n,X;q}$ only depends on the \Kahler moduli. This is the special property characterizing the A-model.

A special role is played by Calabi-Yau 3-folds where the original mirror symmetry is established. In the case of dimension 3,
$$
\dim \bbracket{\overline{M}_{g,n,\beta}(X)}^{vir}=2n
$$

\begin{defn}The Yukawa coupling in the A-model is defined to be the genus 0 3-point correlation function
\begin{eqnarray*}
   {\HH^*(X)}^{\otimes 3}&\to& \C\\
    \alpha\otimes \beta\otimes \gamma &\to& \F^A_{0,3,X;q}\bbracket{\alpha,\beta,\gamma}
\end{eqnarray*}
\end{defn}
If $\beta=0$, we know that the Gromov-Witten invariants are reduced to the classical intersection product
$$
      \abracket{\alpha,\beta,\gamma}_{0,3, \beta=0}=\int_X \alpha\wedge\beta\wedge\gamma
$$
Therefore, the A-model Yukawa coupling
$$
\F^A_{0,3,X;q}\bbracket{\alpha,\beta,\gamma}=\int_X \alpha\wedge\beta\wedge\gamma+\suml_{\beta\neq 0}
q^{\int_\beta \omega_X}\int_{\bbracket{\overline{M}_{0,3,\beta}(X)}^{vir}} ev_1^*\alpha\wedge ev_2^*\beta \wedge ev_3^*\gamma
$$
can be viewed as a quantum deformation of the classical intersection product. Moreover, it gives a q-deformation of the classical ring structure of $\HH^*(X, \C)$, which is called the \emph{quantum cohomology ring}.

\subsection{The B-model and BCOV theory}\label{Intro-B-model}
The geometry of B-model concerns the moduli space of complex structures of Calabi-Yau manifolds. Let $\check{X}_{\tau}$ be a Calabi-Yau 3-fold with nowhere vanishing holomorphic volume form $\Omega_{\check{X}_\tau}$. Let $T_{\check{X}_{\tau}}$ be the holomorphic tangent bundle. Here $\tau$ parametrizes the complex structures.

\begin{defn}
The B-model Yukawa coupling is defined to be
\begin{eqnarray*}
      { \HH^*({\check{X}_{\tau}}, \wedge^*T_{\check{X}_{\tau}})}^{\otimes 3} &\to& \C\\
         \mu_1\otimes \mu_2\otimes \mu_3&\to& F^B_{0,3, \check{X}_{\tau}}\bbracket{\mu_1,\mu_2,\mu_3}= \int_{{\check{X}_{\tau}}}\bracket{\mu_1\wedge \mu_2\wedge \mu_3\vdash \Omega_{\check{X}_{\tau}}}\wedge \Omega_{\check{X}}
\end{eqnarray*}
where $\vdash$ is the natural contraction between tensors in $\wedge^*T_{\check{X}}$ and $\wedge^*T^*_{\check{X}}$.
\end{defn}

Generally speaking, string theory predicts that we should also have the B-model correlation functions
$$
       \F^{B}_{g,n, \check{X}_{\tau}}: \Sym_{\C}^n\bracket{\HH^*(\check{X}_{\tau}, \wedge^*T_{\check{X}_{\tau}})[[t]]}\to \C
$$

We briefly describe the BCOV approach to the B-model correlation functions developed in \cite{Si-Kevin}. Let $\Omega_{\check{X}_{\tau}}$ be a fixed nowhere vanishing holomorphic volume form. The existence of $\Omega_{\check{X}_{\tau}}$ is guaranteed by the Calabi-Yau condition. Let
$$
  \E_{\check{X}_{\tau}}=\PV_{\check{X}_{\tau}}^{*,*}[[t]]
$$
be the space of fields of BCOV theory, where $\PV_{\check{X}_{\tau}}^{*,*}$ is the space of polyvector fields. We define the classical BCOV action as a functional on $ \E_{\check{X}_{\tau}}$ by
$$
        S^{BCOV}=\suml_{n\geq 3}S^{BCOV}_n
$$
where
\begin{eqnarray*}
       S_n^{BCOV}: \Sym\bracket{\E_{\check{X}_{\tau}}^{\otimes n}}&\to&\C\\
       t^{k_1}\mu_1\otimes\cdots\otimes t^{k_n}\mu_n&\to& \int_{\overline{M}_{0,n}}\psi_1^{k_1}\cdots \psi_n^{k_n} \int_{\check{X}_{\tau}}\bracket{\mu_1\cdots\mu_n\vdash \Omega_{\check{X}_{\tau}}}\wedge \Omega_{\check{X}_{\tau}}
\end{eqnarray*}
where $\int_{\overline{M}_{0,n}}\psi_1^{k_1}\cdots \psi_n^{k_n}=\binom{n-3}{k_1,\cdots, k_n}$ is the $\psi$-class integration. Let
$$
     Q=\dbar-t\pa:  \E_{\check{X}_{\tau}}\to  \E_{\check{X}_{\tau}}
$$
be the differential acting on polyvector fields (see Definition \ref{definition-fields}). $Q$ induces a derivation on the space of functionals on $\E_{\check{X}_{\tau}}$, which we still denote by $Q$. Let $\{-.-\}$ be the Poisson bracket on local functionals defined by equation (\ref{Poisson-bracket}). Then $S^{BCOV}$ satisfies the following \emph{classical master equation} \cite{Si-Kevin}
\begin{eqnarray}
     QS^{BCOV}+{1\over 2}\fbracket{S^{BCOV},S^{BCOV}}=0
\end{eqnarray}
The physics meaning of classical master equation is that $S^{BCOV}$ is endowed with a gauge symmetry. $S^{BCOV}$ generalizes the original Kodaira-Spencer gauge action on Calabi-Yau 3-folds \cite{BCOV} to arbitrary dimensions, and remarkably, it also includes the gravitational descendants $t$.

We would like to construct the quantization of the BCOV theory on $\check{X}_{\tau}$, which is given by a family of functionals on $\E_{\check{X}_{\tau}}$ valued in $\C[[\hbar]]$ parametrized by $L>0$
$$
     \F[L]=\suml_{g\geq 0}\hbar^g \F_g[L]
$$
satisfying the axioms of quantization (see definition \ref{axiom-quantization}).

Once we have constructed the quantization $\F[L]$, we can let $L\to \infty$. The quantum master equation at $L=\infty$ says
$$
           Q \F[\infty]=0
$$
This implies that $\F[\infty]$ induces a well-defined functional on the $Q$-cohomology of $\E_{\check{X}_{\tau}}$. We will write
$$
    \F[\infty]=\suml_{g\geq 0}\hbar^g\F_{g, \check{X}_\tau}^B
$$
Using the isomorphism
$$
        \HH^*(\E_{\check{X}_{\tau}}, Q)\iso \HH^*({\check{X}_{\tau}}, \wedge^* T_{{\check{X}_{\tau}}})[[t]]
$$
and decomposing $\F_{g, \check{X}}^B$ into number of inputs, we can define the genus $g$ B-model correlation functions by
$$
    \F_{g,n,\check{X}}^B: \Sym_\C^n\bracket{\HH^*({\check{X}_{\tau}}, \wedge^* T_{{\check{X}_{\tau}}})[[t]]}\to \C
$$

Therefore the problem of constructing higher genus B-model is reduced to the construction of the quantization $\F[L]$, which is controlled by certain $L_\infty$ algebraic structure on the space of local functionals on $\E_{\check{X}_{\tau}}$ \cite{Si-Kevin}. There's an obstruction class for constructing $\F_g[L]$ at each genus $g>0$, and it's natural is conjecture that all the obstruction classes vanish for BCOV theory. For $X$ being one-dimensional, i.e.,  the elliptic curve, this is indeed the case.

To establish mirror symmetry at higher genus, we need to compare the A-model correlation function $\F^A_{g,n,X;q}$ with the B-model correlation function $\F_{g,n,\check{X}_\tau}^B$. In general, $\F_{g,n,\check{X}_\tau}^B$ doesn't depend holomorphically on $\tau$, and there's the famous holomorphic anomalies discovered in \cite{BCOV}. It's also predicted that we should be able to make sense of the limit $\lim\limits_{\bar\tau\to \infty}\F_{g,n,\check{X}_\tau}^B$ around the large complex limit of $\check{X}_{\tau}$. The geometric interpretation of $\lim\limits_{\bar\tau\to \infty}$ is explained in \cite{Si-Kevin} via splitting Hodge filtrations from limiting monodromy filtration.

The higher genus mirror conjecture can be stated as the identification of
$$
F^A_{g,n,X;q} \longleftrightarrow \lim\limits_{\bar\tau\to \infty}\F_{g,n,\check{X}_\tau}^B
$$
under certain identification of cohomology classes
$$
        \HH^*(X,\wedge^* T^*_X) \longleftrightarrow \HH^*(\check{X}_\tau, \wedge^* T_{\check{X}_\tau})
$$
and the mirror map between \Kahler moduli and complex moduli
$$
             q \longleftrightarrow \tau
$$
\subsection{Main result}\label{main-result} Let $\check{X}_\tau=\check{E}_\tau$ be the elliptic curve $\C/\bracket{\Z\oplus \Z \tau}$, where $\tau$ lies in the upper-half plane viewed as the complex moduli of $\check{E}$.

\begin{thm}[\cite{Si-Kevin, open-closed-BCOV}] There exists a unique quantization $\F^{\check{E}_\tau}[L]$ of BCOV theory on $\check{E}_\tau$ satisfying the dilaton equation. Morever, $\F^{\check{E}_\tau}[L]$ satisfies the Virasoro equations.
\end{thm}

Since we know that the A-model Gromov-Witten invariants on the elliptic curve  also satisfies the Virasoro equations \cite{virasoro}, the proof of mirror symmetry can be reduced to the stationary sectors \cite{Hurwitz}. More precisely, let $E$ the dual elliptic curve of $\check{E}$ and $\omega\in H^2(E, \Z)$ be the dual class of a point. The stationary sector of Gromov-Witten invariants are defined for descendants of $\omega$
$$
\abracket{t^{k_1}\omega, \cdots, t^{k_n}\omega}_{g,d, E}= \int_{[\overline{M}_{g,n}(E,d)]^{vir}} \prod_{i=1}^n \psi_i^{k_i}ev_i^*(\omega)
$$

On the other hand, we let $\check{\omega}\in \HH^1(\check{E}_\tau, T_{\check{E}_\tau})$ be the class such that
$\Tr (\check{\omega})=1$. The main theorem in this paper is

\begin{thm} For any genus $g\geq 0$, $n>0$, and non-negative integers $k_1, \cdots, k_n$,
\begin{enumerate}
\item $\F^{\check{E}_\tau}_g[\infty][t^{k_1}\check{\omega},\cdots, t^{k_n}\check{\omega}]$ is an almost holomorphic modular form of weight
$$
\suml_{i=1}^n(k_i+2)=2g-2+2n
$$
It follows that the limit  $\lim\limits_{\bar\tau\to \infty}\F^{\check{E}_\tau}_g[\infty][t^{k_1}\check\omega,\cdots, t^{k_n}\check\omega]
$ makes sense and is a quasi-modular form of the same weight.
\item The higher genus mirror symmetry holds on elliptic curves in the following sense
\begin{eqnarray*}
\sum_{d\geq 0}q^d\left\langle t^{k_1}\omega,\cdots,t^{k_n}\omega \right\rangle_{g,d, E}=\lim\limits_{\bar\tau\to \infty}\F^{E_\tau}_g[\infty][t^{k_1}\check\omega,\cdots, t^{k_n}\check\omega]
\end{eqnarray*}
under the identification $q=\exp(2\pi i\tau)$.
\end{enumerate}
\end{thm}

Section \ref{section-proof} is devoted to prove the above theoerem. Let me sketch the main steps here.  Frist, $\lim\limits_{\bar\tau\to \infty}\F^{E_\tau}_g[\infty][t^{k_1}\check\omega,\cdots, t^{k_n}\check\omega]$ can be expressed in terms of traces of operators from chiral bosonic vertex algebra, using the chirality result proved in Section \ref{section-BCOV}.
 On the other hand,  it's shown in \cite{virasoro} that $\sum\limits_{d\geq 0}q^d\left\langle t^{k_1}\omega,\cdots,t^{k_n}\omega \right\rangle_{g,d, E}$ can be expressed in terms of traces of operators from chiral fermionic vertax algebra, which can be further expressed in terms of bosonic operators using the Boson-Fermion correspondence in the theory of lattice vertex algebra. Finally we show that the bosonic vertex operators obtained from both theories furnish an integrable system which is constrained enough to identify the B-model correlation functions with descendant Gromov-Witten invariants of the mirror.

\addtocontents{toc}{\protect\setcounter{tocdepth}{1}}
\subsection*{Acknowledgement} I'm very grateful to my advisor Shing-Tung Yau for his invaluable support and encouragement during my Ph.D. at Harvard. The current paper is part of my thesis work \cite{Li-thesis} under his supervision. I want to pay special thanks to Kevin Costello, who had kindly explained to me his work on renormalization theory. Part of the work grew out of a joint project with him and was carried out while visiting the mathematics department of Northwestern University. Moreover, I thank Cumrun Vafa for many discussions and his wonderful lectures in string theory. The curiosity of string field theory leads directly to my thesis project. And I  thank Huai-Liang Chang for discussions on deformation theory in moduli problems.
\addtocontents{toc}{\protect\setcounter{tocdepth}{2}}

\section{BCOV theory on elliptic curves}\label{section-BCOV}
\subsection{Classical BCOV theory} The classical BCOV theory on arbitrary Calabi-Yau manifolds has been introduced in \cite{Si-Kevin}. In this section we will focus on elliptic curves, and review some basic geometric settings. We refer to \cite{Si-Kevin} for more detailed discussion.

\subsubsection{Space of fields}
Let $\tau\in \mathbb H$ be in upper half plane, and $\Etau$ be the associated elliptic curve
$$
       \Etau=\C/\bracket{\Z\oplus \Z \tau}
$$
$\tau$ can be viewed as the complex moduli for elliptic curves, which will be fixed in this section. We will write $z$ for the linear coordinate, and let $\Omega$ denote the holomorphic 1-form
$$
   \Omega=dz
$$

Let $\PV_\Etau$ be the space of smooth polyvector fields
$$
           \PV_\Etau=\bigoplus_{i,j}\PV_\Etau^{i,j}, \quad \PV_\Etau^{i,j}={\Omega_\Etau^{0,j}\bracket{ \wedge^i T_{\Etau}^{1,0}}}
$$
$\Omega$ induces an isomorphism between polyvector fields and differential forms
$$
        \PV_\Etau^{i,j}\xrightarrow{\vee \Omega} \Omega_\Etau^{1-i,j}
$$
where $\vee$ is the contraction map. The natural operators $\dbar, \pa$ on differential forms define operators on polyvector fields via the above isomorphism, which we will still denote by $\dbar, \pa$
\begin{align*}
   \dbar: & \PV_\Etau^{i,j}\to \PV_{\Etau}^{i,j+1}\\
    \pa: &\PV_\Etau^{i,j}\to \PV_{\Etau}^{i-1,j}
\end{align*}
This gives $\PV_\Etau$ the structure of differential bi-graded algebra. The cohomology degree of $\PV_\Etau^{i,j}$ is $i+j$.

\begin{defn}\label{definition-fields}
The space of fields of BCOV theory on $\Etau$ is defined to be differential graded algebra
$$
     \E_\Etau=\PV_\Etau[[t]][2]
$$
where $t$ if a formal variable of cohomology degree $2$. The differential is defined by
$$
        Q=\dbar-t\pa
$$
\end{defn}

In the original BCOV approach \cite{BCOV} to Kodaira-Spencer gauge theory, the space of fields is given by the kernel of  $\pa$ and it leads to the problem of non-locality. The space of fields we choose can be viewed as the derived version of the fixed point of $\pa$.  As explained in \cite{Si-Kevin}, $t$ will also play the role of gravitational descendants. This proves to be crucial when we will establish the higher genus mirror symmetry on elliptic curves.

\subsubsection{Space of functionals}
\begin{defn}
The space of functionals $\OO(\E_{\Etau})$ on $\E_{\Etau}$ is defined to be the graded-commutative algebra
$$
   \OO(\E_\Etau)=\prodl_{n\geq 0}\OO^{(n)}(\E_\Etau)=\prodl_{n\geq 0}\Hom\bracket{\E_{\Etau}^{\hotimes n}, \C}_{S_n}
$$
where $\Hom$ denotes the space of continuous linear maps, $\hotimes$ is the completed tensor product, and the
subscript $S_n$ denotes taking $S_n$ coinvariants. Elements of
$\OO^{(n)}(\E_\Etau)$ are said to be of \emph{order} $n$. Given
$S\in \OO(\E_\Etau)$, its component of order $n$ is called
the degree $n$ \emph{Taylor coefficient}, denoted by $D_nS$. The space of derivations on $\OO(\E_\Etau)$ is defined to be
$$
     \Der\bracket{\OO(\E_\Etau)}=\prodl_{n\geq 0}\Hom\bracket{\E_\Etau^{\hotimes}, \E_\Etau}_{S_n}
$$
with the natural pairing
$$
  \Der\bracket{\OO(\E_\Etau)}\times \OO(\E_\Etau)\to \OO(\E_\Etau)
$$
\end{defn}

\begin{defn} A functional $S\in \OO(\E_{\Etau})$ is said to be \emph{local} if $S$ can be written as
$$
    S=\int_M \mathcal L
$$
where $\mathcal L$ is a density-valued poly-differential map on the space of fields, which is called the \emph{lagrangian}. The space of local functionals on $\E_{\Etau}$ is denoted by $\Ol\bracket{\E_\Etau}$.
\end{defn}

There's also a similar definition of local derivatons. We refer to \cite{Costello-book} for more careful definitions.
\\

 One important example local functional is the \emph{Trace map}
\begin{align}
     \Tr: \E_\Etau\to \C, \quad \mu=\suml_{i=0}^\infty \mu_i t^i\to \int_{E\tau} \bracket{\mu_0\vee\Omega}\wedge \Omega
\end{align}
It allows us to turn a local functional $S$ into a derivation $W_S$ by the formula
\begin{align}
       S\bracket{\mu_1,\cdots,\mu_{n-1}, \alpha}\equiv \Tr\bracket{W_S\bracket{\mu_1,\cdots,\mu_{n-1}}\alpha}
\end{align}
for any $\mu_i\in \E_\Etau, \alpha\in \PV_\Etau$.

\begin{defn}
The \emph{Hamiltonian vector field} $V_S$ of a local functional $S$ is defined to be the derivation given by the composite of $W_S$ with $\pa$
\begin{align*}
   \xymatrix{
   \prodl_{n\geq 0}\bracket{\E_\Etau}^{\hotimes}\ar[r]^{W_S}\ar[rd]_{V_S}& \PV_\Etau\ar[d]^{\pa}\\
     &  \PV_\Etau
   }
\end{align*}
We define a \emph{Poisson bracket} on $\Ol\bracket{\E_\Etau}$ by the paring:
\begin{align}\label{Poisson-bracket}
\begin{split}
      \fbracket{,}: \Ol\bracket{\E_\Etau}\otimes \Ol\bracket{\E_\Etau} &\to \Ol\bracket{\E_\Etau}\\
                S_1\otimes S_2 &\to \fbracket{S_1, S_2}=V_{S_1}(S_2)
\end{split}
\end{align}
\end{defn}

In the case of elliptic curves, the Poisson bracket $\fbracket{,}$ has cohomology degree $-3$. Note that in \cite{Si-Kevin}, the same Poisson bracket is constructed via a symplectic formulation following Givental and Coates' work on the A-model and Barannikov's work on the B-model.

\subsubsection{Classical BCOV action}
\begin{defn}
The \emph{classical BCOV action functional} $S^{BCOV}\in \Ol(\E_\Etau)$ is
defined by the Taylor coefficients
$$
         D_n S^{BCOV}\bracket{t^{k_1}\alpha_1,\cdots,
t^{k_n}\alpha_n}=\begin{cases} \abracket{\tau_{k_1}\cdots
\tau_{k_n}}_0
\Tr\bracket{\alpha_1\cdots\alpha_n} & \mbox{if}\ n\geq 3\\
0 & \mbox{if}\ n<3
\end{cases}
$$
where
$$
\abracket{\tau_{k_1}\cdots
\tau_{k_n}}_0=\int_{\overline{M}_{0,n}}\psi_1^{k_1}\cdots
\psi_n^{k_n}=\binom{n-3}{k_1,\cdots, k_n}
$$
is the intersection number of $\psi$-classes on the moduli space of pointed stable rational curves.
\end{defn}

It's shown in \cite{Si-Kevin} that $S^{BCOV}$ satisfies the \emph{classical master equation}
\begin{align}
      QS^{BCOV}+{1\over 2}\fbracket{S^{BCOV}, S^{BCOV}}=0
\end{align}
It defines a dg-structure on $\Ol\bracket{\E_\Etau}$, and the complex
$$
   \bracket{\Ol\bracket{\E_\Etau}, Q+\fbracket{S^{BCOV},-}}
$$
is the \emph{deformation-obstruction complex} controlling the quantization of the classical BCOV theory \cite{Costello-book}.

\subsection{Quantization} We are interested in the perturbative quantization of classical BCOV theory from Costello's renormalization technique \cite{Costello-book}. The general  framework has been discussed in \cite{Si-Kevin} and we briefly review the geometric set-up.

\subsubsection{Regularized \BV operator}
We will fix the standard flat metric on $\Etau$. Let $\dbar^*$ be the adjoint of $\dbar$
$$
       \dbar^*: \PV_\Etau^{i,j}\to \PV_\Etau^{i,j-1}
$$
Let
$$
K_L\in
\Sym^2(\PV_\Etau), \quad L>0
$$
be the heat kernel of the Laplacian $[\dbar,\dbar^*]$, which is normalized by the following equation
$$
     e^{-L[\dbar,\dbar^*]}\alpha=P_1 \Tr(P_2 \alpha), \ \ \forall \alpha \in \PV_\Etau
$$
if we formally write $K_L=P_1\hotimes P_2$.\\

\begin{defn} The \emph{regularized BV operator}
$$
  \Delta_L: \OO^{(n+2)}(\E_\Etau)\to \OO^{(n)}(\E_\Etau)
$$
is defined to be the operator of contracting with the kernel
$$
      \bracket{\pa\otimes 1}K_L
$$
\end{defn}

Using $\Delta_L$, we can define the \emph{regularized \BV bracket}  by
\begin{align}
   \fbracket{S_1, S_2}_L\equiv\Delta_L\bracket{S_1S_2}-\bracket{\Delta_L
S_1}S_2-(-1)^{|S_1|}S_1 \Delta_L S_2
\end{align}
for any $S_1, S_2\in \OO(\E_\Etau)$. It's easy to see that if $S_1, S_2\in \Ol(\E_\Etau)$, then the classical Poisson bracket is recovered from the limit
$$
   \fbracket{S_1, S_2}=\lim_{L\to 0}\fbracket{S_1, S_2}_L
$$

\subsubsection{Regularized propagator}
The propagator of our BCOV theory is defined to be the kernel
\begin{align}
   \PP_\epsilon^L=\int_\epsilon^L \bracket{\dbar^*\pa\otimes 1} K_u du
\end{align}
Let
$$
   \pa_{P_\epsilon^L}: \OO(\Etau)\to \OO(\Etau)
$$
be the operator corresponding to contracting with $P_\epsilon^L$. We have
\begin{align}
   \bbracket{Q, \pa_{P_\epsilon^L}}=\Delta_\epsilon-\Delta_L
\end{align}
Thus, $\pa_{P_\epsilon^L}$ gives a homotopy between the operators $\Delta_L$ and $\Delta_\epsilon$.

\subsubsection{Axioms of quantization}

\begin{defn}[\cite{Si-Kevin}]\label{axiom-quantization}
A quantization of the BCOV theory on $\Etau$ is given by a family of functionals
$$
\F[L] = \suml_{g\geq 0} \hbar^g \F_g[L] \in \OO (\E_\Etau)[[\hbar]]
$$
for each $L \in \R_{> 0}$, which is at least cubic modulo $\hbar$, satisfying
\begin{enumerate}
\item  The renormalization group flow equation
\begin{eqnarray*}
\F[L] = W\bracket{\PP_\eps^L, \F[\epsilon]}
\end{eqnarray*}
for all $L, \eps>0$.
\item The quantum master equation
\begin{eqnarray*}
Q \F[L] + \hbar \Delta_L \F[L] + \frac{1}{2} \{\F[L],\F[L]\}_L = 0,\ \ \forall L>0
\end{eqnarray*}
or equivalently
$$
   \bracket{Q+\hbar \Delta_L}e^{\F[L]/\hbar}=0
$$
\item The locality axiom, as in \cite{Costello-book}. This says that $\F[L]$ has a small $L$ asymptotic expansion in terms of local functionals.
\item The classical limit condition
$$
  \lim\limits_{L\to 0}\F_0[L]=S^{BCOV}
$$

\item Degree axiom. The functional $\F_g$ is of cohomological degree
$$
2(2-2g)
$$

\item
We will give $\E_\Etau$ an additional grading, which we call Hodge weight,  by saying that elements in
$$t^m \Omega^{0,*} (\wedge^k T_{\Etau}) = \PV^{k,*}(\Etau)$$ have Hodge weight $k + m-1$.   We will let $\HW(\alpha)$ denote the Hodge weight of an element $\alpha \in \E_\Etau$. Then, the functional $\F_g$ must be of Hodge weight
$$
2-2g
$$

\iffalse
\item The string equation axiom. Let
$$
         T_{(-1)}: \E\to \E
$$
be the operator defined by
$$
  T_{(-1)} (t^k\mu)=\begin{cases}
         t^{k-1}\mu & \mbox{if}\ k>0\\
         0 & \mbox{if}\ k=0
  \end{cases}
$$
It induces a derivation on $\OO(\E)$, which we still denote by $T_{(-1)}$. Let's define the derivation $\mathcal S\in \Der(\E)$ by
$$
\mathcal S=T_{(-1)} -\dpa{(1)}
$$
Then the string equation axiom asserts that there exists $K[L]\in \OO(\E)[[\hbar]]$ such that
\begin{eqnarray*}
         \mathcal S \F[L]= Q K[L]+ \{\F[L], K[L]\}_L+\hbar \Delta_L K[L]
\end{eqnarray*}

\fi

\end{enumerate}
\end{defn}

All the above properties of $\F[L]$ are motivated by mirror symmetry and modeled on the corresponding Gromov-Witten theory from the A-model.

\subsubsection{Dilaton equation and string equation}
We can also couple the BCOV theory with dilaton equation and string equation. But we will only do this in the homotopic sense \cite{Si-Kevin}. Let
$$
   \Eu: \OO(\E_\Etau)\to \OO(\E_\Etau)
$$
be the Euler vector field, defined by
$$
   \Eu \Phi=n\Phi
$$
if $\Phi\in \OO^{(n)}(\E_\Etau)$ is of order n.  Let $t\cdot 1\in t\PV_{\Etau}^{0,0}$ be the constant polyvector field with coefficient $t$. It naturally associates a derivation ${\pa\over \pa \bracket{t\cdot 1}}\in \Der(\E_\Etau)$.  We define the dilaton vector field $\pa_{Dil}$ by
$$
   \pa_{Dil}={\pa\over \pa \bracket{t\cdot 1}}-\Eu
$$

\begin{defn}[Dilaton axiom]
A quantization of the BCOV theory $\fbracket{\F[L]}$ satisfies the \emph{dilaton equation} if there exists a family of functionals $\G[L]\in \hbar \OO(\E_\Etau)[[\hbar]]$ such that
$$
 \bracket{Q+\hbar\Delta_L+\delta\bracket{\pa_{Dil}-2\hbar{\pa\over \pa \hbar}}}e^{F[L]/\hbar+\delta\G[L]/\hbar}=0
$$
where $\delta$ is a formal variable of cohomology degree zero with $\delta^2=0$. Moreover, we require the following renormalization group flow equation
$$
   \F[L]+\delta \G[L]=W\bracket{\PP_\epsilon^L, \F[\epsilon]+\delta \G[\epsilon]}
$$

\end{defn}

The dilaton axiom can be equivalently stated by saying that the family of quantization $\F[L]+\delta \bracket{\pa_{Dil}-2\bracket{\hbar{\pa\over \pa \hbar}-1}}\F[L]$ over $\C[\delta]/\delta^2$ is homotopic to the trivial family $\F[L]$ \cite{Si-Kevin}.

\begin{thm}[\cite{Si-Kevin}]\label{theorem-quantization}
There exists a unique quantization (up to homotopy) of the BCOV theory satisfying the dilaton equation on any elliptic curve.
\end{thm}

Thus, we will use $\fbracket{\F^{\Etau}[L]}$ to denote such a quantizaton with dilaton axiom on the elliptic curve $\Etau$ without ambiguity.

Next, we discuss the homotopy version of string equation.  Let
$$
  t^{-1}\star: \E_{\Etau}\to \E_{\Etau}
$$
be the operator defined by
$$
     t^{-1}\star \bracket{t^{k}\mu}=\begin{cases} t^{k-1}\mu & \text{if}\ k>0 \\ 0 & \text{if}\ k=0 \end{cases}
$$
It induces a derivation on $\OO(\E_\Etau)$, which will be denoted by $\pa_{t^{-1}\star}$. We define another linear operator on $\E_\Etau$ depending on the scale $L$
$$
   Y[L] (t^k \mu)=\begin{cases}0 & \text{if}\ k>0 \\ \int_0^L du \dbar^*\pa e^{-u[\dbar,\dbar^*]} \mu & \text{if}\ k=0\end{cases}
$$
and use the same symbol $Y[L]$ to denote its induced derivation on $\OO(\E_\Etau)$. The string vector field depends on the scale and is defined by
$$
  \pa_{Str}[L]={\pa\over \pa (1)}-\pa_{t^{-1}\star}+Y[L]
$$
where ${\pa\over \pa (1)}$ is the derivation of contracting with the constant polyvector field $1$. Let $S^{BCOV}_3$ be the cubic Taylor coefficient of the classical BCOV functional. ${\pa\over \pa (1)}S_3^{BCOV}$ will be the quadratic local functional given by the trace pairing.

\begin{defn}[String equation]
A quantization of the BCOV theory $\fbracket{\F[L]}$ satisfies the \emph{string equation} if there exists a family of functionals $\G[L]\in \hbar \OO(\E_\Etau)[[\hbar]]$ such that
$$
 \bracket{Q+\hbar\Delta_L+\delta\bracket{\pa_{Str}[L]-{1\over \hbar}{\pa\over \pa (1)}S_3^{BCOV}}}e^{F[L]/\hbar+\delta\G[L]/\hbar}=0
$$
where $\delta$ is a formal variable of cohomology degree zero with $\delta^2=0$. Moreover, we require the following renormalization group flow equation
$$
   \F[L]+\delta \G[L]=W\bracket{\PP_\epsilon^L, \F[\epsilon]+\delta \G[\epsilon]}
$$
\end{defn}

The reason for the choice of the scale-dependent string operator is that it is compatible with the renormalization group flow equation and quantum master equation in the following sense
\begin{align*}
      &\bracket{\pa_{Str}[L]-{1\over \hbar}{\pa\over \pa (1)}S_3^{BCOV}}e^{\hbar \pa_{\PP_{\epsilon}^L}}=e^{\hbar \pa_{\PP_\epsilon^L}}\bracket{\pa_{Str}[\epsilon]-{1\over \hbar}{\pa\over \pa (1)}S_3^{BCOV}}\\
      &\bbracket{\bracket{\pa_{Str}[L]-{1\over \hbar}{\pa\over \pa (1)}S_3^{BCOV}}, Q+\hbar\Delta_L }=0
\end{align*}
At the classical level, it simply says that
$$
   \pa_{Str}[0]S^{BCOV}={\pa\over \pa (1)}S_3^{BCOV}
$$
which is the familiar string equation.

\begin{thm}[\cite{Si-Kevin}]
The quantization $\fbracket{\F^{\Etau}[L]}$ on the elliptic curve satisfies the string equation.
\end{thm}

\subsection{Higher genus B-model} In this section we will explain the establishment of the higher genus B-model from the quantization of BCOV theory. The general idea is presented in \cite{Si-Kevin}. We will focus on elliptic curves and discuss the relevant geometry needed for higher genus mirror symmetry.

\subsubsection{B-model correlation functions} Let $\fbracket{\F^{\Etau}[L]}$ be the unique quantization on the elliptic curve $\Etau$ satisfying the dilaton equation. Since $\Etau$ is compact, the kernel $\PP_L^\infty$ is in fact a smooth kernel. Thus the renormalization group equation allows us to take the following limit
$$
    \F^{\Etau}[\infty]\equiv \sum_{g\geq 0}\hbar^g\F_g^{\Etau}[\infty]\equiv\lim_{L\to \infty}\F[L] \in \OO(\E_\Etau)[[\hbar]]
$$
On the other hand, the BV operator $\Delta_L$ is the contraction with the kernel representing the operator $\pa e^{-L[\dbar,\dbar^*]}$, which vanishes at $L=\infty$. The quantum master equation at $\infty$ implies that
$$
       Q \F^{\Etau}_g[\infty]=0, \quad \forall g\geq 0
$$
Therefore $\F_g^{\Etau}[\infty]$ descends a formal function on the cohomology $\HH^*(\E_\Etau, Q)$.

Recall that in constructing $\F^{\Etau}[L]$, we need a choice of a metric (we have chosen the standard flat metric for simplicity) as the gauge-fixing condition and construct the heat kernel.  However
\begin{lem}
The partition function $\F_g^{\Etau}[\infty]$, when viewed as a formal function on $\HH^*(\E_\Etau, Q)$, doesn't depend on the choice of the metric.
\end{lem}
This is precisely the geometric meaning of quantum master equation. See \cite{Si-Kevin}*{Section 4} for a detailed proof.

However, the choice of the metric allows us to identify $\HH^*(\E_\Etau, Q)$ with the sheaf cohomology of polyvector fields. In fact, using Hodge theory, it's easy to prove the following

\begin{lem}\label{isomorphism-polyvectors} With a fixed choice of metric,  there are natural isomorphisms
$$
    \HH^*(\Etau, \wedge^* T_{\Etau})[[t]]\iso \mathbb H^*(\Etau \wedge^* T_{\Etau})[[t]]\iso \HH^*(\E_\Etau,Q)
$$
where $\HH^*(\Etau, \wedge^* T_{\Etau})$ is the sheaf cohomology of $ \wedge^* T_{\Etau}$ on $\Etau$, and $ \mathbb H^*(\Etau \wedge^* T_{\Etau})$ is the space of Harmonic polyvector fields .
\end{lem}

\begin{defn}\label{higher-genus-B-invariants}
The genus $g$ correlation function of the quantization $\F^\Etau$ is defined to be
$$
   \abracket{t^{k_1}\mu_1,\cdots, t^{k_n}\mu_n}_{g,n}^{\F^\Etau}=\bracket{{\pa\over \pa \bracket{t^{k_1}\mu_1}}\cdots {\pa\over \pa \bracket{t^{k_n}\mu_n}}\F^{\Etau}_g[\infty]}(0)
$$
for $\mu_i\in \mathbb H^*(\Etau, \wedge^* T_{\Etau})$.
\end{defn}
This will be the candidate of higher genus B-model invariants mirror to the descendant Gromov-Witten invariants.

\subsubsection{Dilaton equation and string equation}
The quantization $\F^{\Etau}$ satisfies the homotopic dilaton equation and string equation. This will imply that the correlation functions satisfy the strict dilaton and string equations.

\begin{prop} The correlation functions $\abracket{-}_{g,n}^{\F^\Etau}$ satisfies the  dilaton equation
$$
   \abracket{t, t^{k_1}\mu_1,\cdots, t^{k_n}\mu_n}_{g,n+1}^{\F^\Etau}=(2g-2+n)\abracket{t^{k_1}\mu_1,\cdots, t^{k_n}\mu_n}_{g,n}^{\F^\Etau}
$$
and the string equation
$$
       \abracket{1, t^{k_1}\mu_1,\cdots, t^{k_n}\mu_n}_{g,n+1}^{\F^\Etau}=\sum_{i}  \abracket{t^{k_1}\mu_1,\cdots, t^{k_i-1}\mu_i, \cdots t^{k_n}\mu_n}_{g,n}^{\F^\Etau}, \quad \forall 2g+n\geq 3
$$
for any $\mu_i\in \mathbb H^*(\Etau, \wedge^* T_{\Etau})$.
\end{prop}
\begin{proof} The dilaton axiom at $L\to\infty$ says that
$$
   \bracket{\pa_{Dil}-2\bracket{\hbar{\pa\over \pa \hbar}-1}}\F^{\Etau}[\infty]=Q\G[\infty]
$$
If we restrict to $Q$-cohomology classes, we find
$$
     \bracket{\pa_{Dil}-(2g-2)}\F_g^{\Etau}[\infty]=0 \quad \text{on}\ \HH^*(\E_\Etau, Q)
$$
This proves the dilaton equation. Similar argument proves the string equation.
\end{proof}

\subsection{Chirality} Two dimensional quantum field theory usually behaves better than higher dimensions.  In this section, we show that the BCOV theory on elliptic curves is divergence free and its quantization can be built up using local functionals with holomorphic derivatives only.
\subsubsection{Translation invariant local functionals} The local functionals on $\E_\Etau$ can be described using D-module \cite{Costello-book}. Let $J(\E_\Etau)$ be the $D_{\Etau}$-module of smooth jets of $\PV_{\Etau}$ on the elliptic curve $\Etau$. If we choose coordinates on a domain $U$ on $\Etau$, then the sections of $J(\E_\Etau)$ are given by
$$
    \cinfty\bracket{U}\otimes \C[[z,\bar z, d\bar z, \pa_z, t]][2]
$$
The space of local functionals on $\E_\Etau$ can be described by
$$
   \Ol\bracket{\E_\Etau}=C_{red}^*\bracket{J(\E_\Etau)[-1]}\otimes_{D_{\Etau}} \Omega^2_{\Etau}[2]
$$
where $C_{red}^*\bracket{J(\E_\Etau)[-1]}$ is the reduced \CE complex taken in the symmetric monoidal category of $D_{\Etau}$ modules equipped with tensor product over $\cinfty_{\Etau}$, and $\Omega^2_{\Etau}$ is the right $D_{\Etau}$-module of top differential forms on $\Etau$.

$\Etau$ is equipped with translation transformations. We can also consider translation invariant local functionals. Let
$$
   J(\E_\Etau)^{\Etau}=\C[[z, \bar z, d\bar z , \pa_z, t]][2]
$$
be the space of translation invariants jets,
$$
   D=\C\bbracket{{\pa\over \pa z}, {\pa\over \pa \bar z}}\subset D_{\Etau}
$$
be the subspace of translation invariant differential operators in $D_\Etau$, and
$$
   \Omega^*=\C\bbracket{dz, d\bar z}
$$
be the translation invariant differential forms.  $J(\E_\Etau)^{\Etau}$ has a natural module structure over $D$. We will use $ \Ol\bracket{\E_\Etau}^{\Etau}$ to denote the space of translation invariant local functionals. Then
$$
     \Ol\bracket{\E_\Etau}^{\Etau}=C_{red}^*\bracket{ J(\E_\Etau)^{\Etau}[-1]}\otimes_{D}  \Omega^2[2]
$$

The dg structure of both $\Ol\bracket{\E_\Etau}$ and $\Ol\bracket{\E_\Etau}^{\Etau}$ are induced from the classical BCOV action
$$
   Q+\fbracket{S^{BCOV},-}
$$
The corresponding complex controls the quantization of BCOV theory on the elliptic curve \cite{Si-Kevin}.

\subsubsection{Chiral local functionals}
\begin{defn}
A local functional is said to be \emph{chiral} if it only contains holomorphic derivatives.
\end{defn}

We would like to consider translation invariant chiral local functionals. Let
$$
   D^{hol}=\C\bbracket{\pa\over \pa z}\subset D
$$
be the translation invariant holomorphic differential operators. We define a subcomplex of $J(\E_{\Etau})^\Etau[-1]$ by
$$
    \g^{hol}_{\E_\Etau}=\C[[z, \pa_z,t]][1]\subset J(\E_{\Etau})^\Etau[-1]
$$
Since $S^{BCOV}$ is chiral, it's easy to see that $\g^{hol}_{\E_\Etau}$ becomes a sub $L_\infty$ algebra. Moreover, the inclusion
$$
   \g^{hol}_{\E_\Etau}\into J(\E_{\Etau})^\Etau[-1]
$$
is quasi-isomorphic.

\begin{lem}\label{chiral-functional}
The translation invariant obstruction complex $\Ol\bracket{\E_\Etau}^{\Etau}$ is quasi-isomorphic to
$$
   \bracket{ C_{red}^*\bracket{ \g^{hol}_{\E_\Etau}}\otimes_{D^{hol}} \C dz[2]}\oplus \bracket{ C_{red}^*\bracket{ \g^{hol}_{\E_\Etau}}\otimes_{D^{hol}} \C dzd\bar z[2]}
$$
\end{lem}
\begin{proof} From Koszul resolution, $\Ol\bracket{\E_\Etau}^{\Etau}$ is quasi-isomorphic to de Rham complex for the $D_{\Etau}$-module $C_{red}^*\bracket{J(\E_\Etau)[-1]}$
   \begin{align*}
      \Ol\bracket{\E_\Etau}^{\Etau}&= C_{red}^*\bracket{J(\E_\Etau)^{\Etau}[-1]}\otimes_{D}\C dz d\bar z[2]
      \\&\simeq  \Omega^*\bracket{C_{red}^*\bracket{J(\E_\Etau)^{\Etau}[-1]}}[2]\\
         &\simeq \Omega^*\bracket{C_{red}^*\bracket{ \g^{hol}_{\E_\Etau}}}[2]\\
         &\simeq \C[d\bar z]\otimes \Omega_{hol}^{*}\bracket{C_{red}^*\bracket{ \g^{hol}_{\E_\Etau}}}[2]\\
         &\simeq \C[d\bar z]\otimes  \bracket{{C_{red}^*\bracket{ \g^{hol}_{\E_\Etau}}}\otimes_{D^{hol}}\C dz[2]}
   \end{align*}
\end{proof}

We describe the first component in the above decomposition as a subcomplex of $\Ol\bracket{\E_\Etau}^{\Etau}$. Let
$$
   \wedge_{d\bar z}: \E_\Etau\to \E_\Etau, \quad \mu\to d\bar z\wedge \mu
$$
It induces a derivation on functionals which we still denote by $\wedge_{d\bar z}$. This gives an embedding of complexes
$$
\wedge_{d\bar z}: C_{red}^*\bracket{ \g^{hol}_{\E_\Etau}}\otimes_{D^{hol}} \Omega^2[1]\into \Ol\bracket{\E_\Etau}^{\Etau}
$$
\begin{defn}
A translation invariant chiral local functional is said to be \emph{admissible} if it lies in the image of $\wedge_{d\bar z}$ in the above embedding.  The space of admisslbe local functionals will be denoted by $\Ol^{ad}\bracket{\E_\Etau}$.
\end{defn}

It's easy to check that the first component $ \bracket{ C_{red}^*\bracket{ \g^{hol}_{\E_\Etau}}\otimes_{D^{hol}} \C dz[2]}$ in Lemma \ref{chiral-functional} is precisely transgressed to admissible local functionals.  A functional is admissible if it's chiral with precisely one input containing $d\bar z$ and totally symmetric with the position of $d\bar z$. The following lemma is easy to check

\begin{lem}
Let  $S_1, S_2$ be two admissible functionals, then
\begin{enumerate}
\item $\dbar S_1=\dbar S_2=0$, and
\item $\fbracket{S_1, S_2}$ is also admissble.
\item The classical BCOV action $S^{BCOV}$ is admissible.
\end{enumerate}
\end{lem}

\subsubsection{Quantization with chiral local functionals}

We analyze in detail the quantization with chiral local functionals in this section. We show that the quantization of BCOV theory on the elliptic curve  is given by
 local chiral functionals as a quantum correction of the classical BCOV action. Since the problem is local, we will work on $\C$. The space of fields will be
$$
   \E_\C= \PV^{*,*}_{\C,c}[[t]]
$$
where $\PV^{*,*}_{\C,c}$ is the compactly supported polyvector fields on $\C$. The classical BCOV action  and the translation-invariant admissible local functionals on $\C$, which will be denoted by $S^{BCOV}$ and $\Ol^{ad}\bracket{\E_{\C}}$ respectively, are the same as that on the elliptic curve.  The regularized BV operator $\Delta_L$ on $\C$ is the contraction with
$$
  {1\over 4\pi L}\bracket{{\bar z_1-\bar z_2}\over 4t}e^{-|z_1-z_2|^2/4t}\bracket{d\bar z_1\otimes 1+1\otimes d\bar z_2}
$$
and the regularized propagator is
$$
   \PP_\epsilon^L=\int_\epsilon^L {dt \over 4\pi L}\bracket{{\bar z_1-\bar z_2}\over 4t}^2e^{-|z_1-z_2|^2/4t}
$$
\subsubsection*{Renormalization group flow}
\begin{prop}\label{finiteness}
Let $S=\suml_{g\geq 0}\hbar^g S_g\in \Ol^{ad}(\E_\C)[\hbar]$, then for any $L>0$, the following limit exists
\begin{align}\label{effective-functional}
     S[L]=\lim_{\epsilon\to 0}\hbar \log\bracket{\exp\bracket{\hbar\pa_{\PP_\epsilon^L}} \exp\bracket{S/\hbar}}\in \OO\bracket{\E_\C}[[\hbar]]
\end{align}
and $F[L]$ satisfies the renormalization group flow equation
$$
    S[L]=W\bracket{\PP_\epsilon^L, F[\epsilon]}, \quad \forall \epsilon, L>0
$$
\end{prop}
\begin{proof}This follows from a finiteness result for graph integrals  in \cite[Prop B.1]{Li-modular}.
\end{proof}
This proposition says that if we build up the quantum theory using admissible functionals, then the counter-terms vanish. \\

\subsubsection*{Quantum master equation}$\empty$\\

We will see how the quantum master equation will look like for admissible functionals.  Let $S=\suml_{g\geq 0}\hbar^g S_g\in \Ol^{ad}(\E_\C)[\hbar]$. By Proposition \ref{finiteness}, we obtain a family of functionals $\fbracket{S[L]}_{L>0}$ satisfying renormalization group equation. The quantum master equation is given by
$$
     \bracket{Q+\hbar \Delta_L}e^{S[L]/\hbar}=0
$$
By construction,
\begin{align*}
       \bracket{Q+\hbar \Delta_L}e^{S[L]/\hbar}&=\lim_{\epsilon\to 0}\bracket{Q+\hbar \Delta_L}e^{\hbar \pa_{P_\epsilon^L}}e^{S/\hbar}\\
       &=\lim_{\epsilon\to 0}e^{\hbar \pa_{\PP_\epsilon^L}}\bracket{Q+\hbar \Delta_\epsilon}e^{S/\hbar}\\
       &=\lim_{\epsilon\to 0}\hbar e^{\hbar \pa_{\PP_\epsilon^L}}\bracket{-(t\pa) S+{1\over 2}\fbracket{S, S}_\epsilon}e^{S/\hbar}
\end{align*}

\begin{lem-defn} \label{deformed BV}
Let $S_1, S_2\in \Ol^{ad}(\E_\C)[[\hbar]]$ be two admissible local functionals, then the following limit exist
\begin{align}
    \fbracket{S_1, S_2}^\prime\equiv\lim_{\epsilon\to 0} e^{\hbar {\pa\over \pa \PP_\epsilon^L}}\Delta_\epsilon\bracket{S_1, S_2}
\end{align}
as an element in $\Ol^{ad}(\E_\C)[[\hbar]]$. It doesn't depend on $L$, and it can be viewed as the quantum deformation of the classical Poisson bracket.
\end{lem-defn}
\begin{proof} This follows from in \cite[Prop B.2]{Li-modular}.
\end{proof}

This Lemma implies that the obstruction of quantization with admissible local functionals is still admissible.

\begin{prop}
Let $S\in \Ol^{ad}(\E_\C)[[\hbar]]$, and $\fbracket{S[L]}_{L}$ be the effective functional defined by \eqref{effective-functional}. Then $\fbracket{S[L]}_{L}$ satisfies the quantum master equation if and only if
\begin{align} \label{deformed-QME}
   QS+{1\over 2}\fbracket{S, S}^\prime=0
\end{align}
\end{prop}
\begin{proof}
By the above lemma,
\begin{align*}
 \bracket{Q+\hbar \Delta_L}e^{S[L]/\hbar}=\hbar e^{\hbar \pa_{P_0^L}}\bracket{QS+{1\over 2}\fbracket{S,S}^\prime}e^{S/\hbar}
\end{align*}
The proposition now follows.
\end{proof}

\begin{rmk}
Equation (\ref{deformed-QME}) can be viewed as quantum corrected equation for the classical master equation. The classical BV bracket contains a single contraction between two local functionals, and the quantization deforms the BV bracket to include all multi-contractions. In fact, using the formula in \cite[Prop B.2]{Li-modular}, it's easy to see that $\fbracket{S, S}^\prime$ is precisely the OPE for certain fields in conformal field theory. This point of view will be addressed in a separate paper.
\end{rmk}

\subsubsection*{Quantization of BCOV theory on elliptic curves}

\begin{thm}\label{admissible-quantization}
There exists an admissible functional $S=\sum\limits_{g\geq 0}S_g \in \Ol^{ad}\bracket{\E_\Etau}[[\hbar]]$ with $S_0=S^{BCOV}$ satisfying
$$
   QS+{1\over 2}\fbracket{S, S}^\prime=0
$$
such that the quantization of BCOV theory $\F^{\Etau}[L]$ is given by

\begin{align*}
     \F^{\Etau}[L]=\lim_{\epsilon\to 0}\hbar \log\bracket{\exp\bracket{\hbar\pa_{\PP_\epsilon^L}} \exp\bracket{S/\hbar}}
\end{align*}

\end{thm}
\begin{proof} It follows from Lemma \ref{chiral-functional} and Lemma/Definition \ref{deformed BV} that the obstruction complex of the quantization can be equivalently described using admissible functionals

$$
   \bracket{\Ol^{ad}\bracket{\E_\Etau}, Q+\fbracket{S^{BCOV},-}}
$$

Then the theorem follows from Theorem \ref{theorem-quantization}.
\end{proof}

\subsection{Conformal symmetry}
In the previous section, we find a local description of the quantization $\F^{\Etau}[L]$ given by chiral local functionals. By Proposition \ref{finiteness}, the quantum theory is in fact divergence free. This implies that we can impose certain conformal symmetry for the quantum theory.

Let's consider the translation invariant BCOV theory on $\C$. The space of fields is again denoted by
$$
  \E_{\C}=\PV_{\C,c}^{*,*}[[t]]
$$
 Let $R_\lambda$ be the following rescaling operator on fields
$$
  R_\lambda\bracket{t^k \alpha(z,\bar z)d\bar z^n\pa_z^m}=\lambda^{n-m}t^k\alpha(\lambda z, \lambda \bar z)d\bar z^n \pa_z^m, \quad \lambda\in \R^+
$$
It induces a rescaling operation on the functionals given by
$$
    R_\lambda^*I[\mu_1,\cdots, \mu_n]=I[R_{\lambda^{-1}}\mu_1,\cdots, R_{\lambda^{-1}}\mu_n]
$$
for $I\in \OO\bracket{\E_{\C}}, \mu_i\in \E_{\C}$.

\begin{prop} Let $\fbracket{\F[L]}_{L>0}$ be a family of effective actions satisfying renormalization group flow and quantum master equation, then
$$
     \F_{\lambda}[L]\equiv \lambda^{2\hbar{\pa\over \pa \hbar}-2}R_\lambda^*\bracket{\F[\lambda^2 L]}
$$
also satisfies the renormalization group flow and quantum master equation.
\end{prop}
\begin{proof}
This is a straight-forward computation.
\end{proof}
Let's consider the effective functional constructed from admissible local functionals by \eqref{effective-functional}
\begin{align*}
     \F[L]=\lim_{\epsilon\to 0}\hbar \log\bracket{\exp\bracket{\hbar\pa_{P_\epsilon^L}} \exp\bracket{S/\hbar}}\in \OO\bracket{\E_{\C}}[[\hbar]]
\end{align*}
where $S=\sum\limits_{g\geq 0}S_g\hbar ^g\in \Ol^{ad}(\E_{\C})[[\hbar]]$. The rescaling operation on the effective functional $\F[L]$ induces a rescaling operation on the local functional $S$. In fact,
$$
    \lim_{L\to 0}\F_{\lambda}[L]=\lambda^{2\hbar{\pa\over \pa \hbar}-2}R_{\lambda}^*S
$$
Therefore we can restrict to the quantization which is fixed under the $\lambda$-rescaling. This imposes the further condition
\begin{align}
    R_{\lambda}^*S_g=\lambda^{2-2g}S_g
\end{align}

\begin{prop}\label{holomorphic derivatives}
Let $S=\sum\limits_{g\geq 0}\hbar^g S_g=\lim\limits_{L\to 0}\F^{\Etau}[L]$ be the admissible functional constructed in Theorem \ref{admissible-quantization}, then $S_g$ contains precisely $2g$ holomorphic derivatives.
\end{prop}
\begin{proof} It follows from the condition
$$
R_{\lambda}^*S_g=\lambda^{2-2g}S_g
$$
and the structure of the admissible local functionals.
\end{proof}

\subsection{Virasoro equations}
We will discuss in this section the Virasoro symmetries for the quantization $\F^\Etau[L]$ on the elliptic curve $\Etau$ proved in \cite{Si-Kevin}. This can be viewed as the mirror equations for the Virasoro constraints of Gromov-Witten invariants on elliptic curves, first discovered by \cite{Virasoro-GW}, and  proved by \cite{virasoro} in general.

We define the following operators $E_m, Z_m$ for $m\geq -1$. If $m\geq 0$, then
\begin{eqnarray*}
       E_m:  \PV^{i,j}_\Etau[[t]]&\to& \PV^{i,j}_\Etau[[t]]\\
             t^k \alpha &\to& t^{m+k}\bracket{k+i}_{m+1} \alpha\\
        Z_m:      \PV^{i,j}_\Etau[[t]]&\to& \PV^{i,j+1}_\Etau[[t]]\\
           t^k \alpha&\to&  t^{m+k}\bracket{k+i}_{m+1}d\bar z\wedge \alpha
\end{eqnarray*}
where $(n)_m=n(n+1)\cdots (n+m-1)$ is the Pochhammer symbol. For $m=-1$, we have
\begin{eqnarray*}
           E_{-1}:  \PV^{i,j}_\Etau[[t]]&\to& \PV^{i,j}_\Etau[[t]]\\
             t^k \alpha &\to& \begin{cases} t^{k-1} \alpha& k>0 \\ 0 & k=0  \end{cases}\\
        Z_{-1}:      \PV^{i,j}_\Etau[[t]]&\to& \PV^{i,j+1}_\Etau[[t]]\\
           t^k \alpha&\to&  \begin{cases}t^{k-1} d\bar z\wedge \alpha & k>0 \\ 0 & k=0 \end{cases}
\end{eqnarray*}

Both $E_m$ and $Z_m$ naturally induce the operators acting on $\OO(\Etau)$, which we denote by the same symbols.

\begin{defn} We define the effective Virasoro operators $\{\mathcal L_m[L], \mathcal D_m[L]\}_{m\geq -1}$\begin{enumerate}
\item If $m\geq 0$, then
\begin{eqnarray*}
     \mathcal L_m[L]=-(m+1)!\dpa{(1\cdot t^{m+1})}+E_m
\end{eqnarray*}
and also
\begin{eqnarray*}
     \mathcal D_m[L]=-(m+1)!\dpa{(d\bar z\cdot t^{m+1})}+Z_m
\end{eqnarray*}
which doesn't depend on the scale $L$.
\item If $m=-1$, the operators $\mathcal L_{-1}[L]$ will depend on the scale $L$. Let $Y[L]$ be the operator
\begin{eqnarray*}
     Y[L][\alpha]=\begin{cases}
     \int_0^L du \dbar^*\pa e^{-uH}\alpha & \alpha\in \PV_\Etau^{*,*}\\
     0 &\alpha\in t \PV_\Etau^{*,*}[[t]]
          \end{cases}
\end{eqnarray*}
and $\tilde Y[L]$ be the operator
\begin{eqnarray*}
     \tilde Y[L][\alpha]=\begin{cases}
     \int_0^L du \dbar^*\pa e^{-uH}( d\bar z\wedge\alpha) & \alpha\in \PV_\Etau^{*,*}\\
     0 &\alpha\in t \PV_\Etau^{*,*}[[t]]
          \end{cases}
\end{eqnarray*}
Recall that $S_3^{BCOV}\in \Sym^3(\E^\vee)$ is the local functional given by the order three component of the classical BCOV action.  Then we define $\mathcal L_{-1}[L]$  by
$$
        \mathcal L_{-1}[L]=-\dpa{(1)}+E_{-1}-Y[L]+{1\over \hbar}\dpa{(1)}S_3^{BCOV}
$$
and $\mathcal D_{-1}[L]$  by
$$
    \mathcal D_{-1}[L]=-\dpa{(d\bar z)}+Z_{-1}-\tilde Y[L]+{1\over \hbar}\dpa{(d\bar z)}S_3^{BCOV}
$$
Note that $\dpa{(1)}S_3^{BCOV}$ is precisely the Trace pairing.
\end{enumerate}

\end{defn}

\begin{lem}
The operators $\{\mathcal L_m[L], \mathcal D_m[L]\}_{m\geq -1}$ satisfy the Virasoro relations
\begin{eqnarray*}
          \bbracket{\mathcal L_m[L], \mathcal L_n[L]}&=&(m-n)\mathcal L_{m+n}[L]\\
           \bbracket{\mathcal L_m[L],\mathcal D_n[L]}&=&(m-n)\mathcal D_{m+n}[L]\\
           \bbracket{\mathcal D_m[L], \mathcal D_n[L]}&=&0
\end{eqnarray*}
for all $m,n\geq -1$ and for any $L$.
\end{lem}

\begin{thm}[\cite{Si-Kevin}]\label{BCOV-virasoro-equations}
The quantization $\F^\Etau[L]$ of BCOV theory on the elliptic curve at $L=\infty$ satisfies the following Virasoro equations
\begin{align*}
    \mathcal L_m[\infty] e^{\F^\Etau[\infty]/\hbar}=\mathcal D_m[\infty] e^{\F^\Etau[\infty]/\hbar}=0 \ \ \mbox{on}\ \ \HH^*\bracket{\E_\Etau, Q}
\end{align*}
for any $m\geq -1$.
\end{thm}

\section{Higher Genus Mirror Symmetry on elliptic curves}\label{section-proof}

Mirror symmetry is a duality between symplectic geometry of Calabi-Yau manifolds (A-model) and complex geometry of the mirror Calabi-Yau manifolds (B-model). In the one-dimension case, i.e. elliptic curves, the mirror map is simple to describe. Let $E$ represent an elliptic curve. In the A-model, we have the moduli of (complexified) K\"{a}hler class $[\omega]\in \HH^2(E, \C)$ parametrized by the (complexified) symplectic volume
$$
    q=\int_E \omega
$$

In the B-model, we have the moduli of inequivalent complex structures which is identified with $\mathbb H/SL(2,\Z)$. Here  $\mathbb H$ is the upper-half plane, and we represent the elliptic curve $\Etau$ as $\C/ (\Z\oplus \Z \tau)$ and identify $\tau$ in $\mathbb H$ under the modular transformation
\begin{align*}
   \tau\to {A\tau+B\over C\tau+D}, \ \ \mbox{for}\ \gamma\in \begin{pmatrix} A& B\\ C& D \end{pmatrix} \in SL(2, \Z)
\end{align*}

The mirror map simply identifies the pair $(E, q)$ with the pair $(\Etau, \tau)$ via
\begin{align*}
             q=e^{2\pi i \tau}
\end{align*}
and mirror symmetry predicts the equivalence between the Gromov-Witten theory of $E$ in the A-model and certain quantum invariants of $\Etau$ in the B-model. We will show in this section that the quantum invariants in the B-model are precisely BCOV invariants constructed from the quantization of the classical BCOV action in the previous section. We prove that the BCOV invariants can be identified with the generating function of descendant Gromov-Witten invariants of the mirror elliptic curve, to all genera.  This established the higher genus mirror symmetry on elliptic curves, as originally proposed in \cite{BCOV}. More precisely, let
$$
   \tilde \omega\in \HH^2(E, \C)
$$
be the class of the Poincare dual of a point. Let $k_1,\cdots, k_n$ be non-negative integers. We consider the following generating function of descendant Gromov-Witten invariants
\begin{align}\label{GW partition}
       \sum_{d\geq 0}q^d\left\langle \prod_{i=1}^n\tau_{k_i}(\tilde \omega) \right\rangle_{g,d}=\sum_{d\geq 0}
       \int_{[\overline{M}_{g,n}(E,d)]^{vir}} \prod_{i=1}^n \psi_i^{k_i}ev_i^*(\tilde \omega)
\end{align}
where $\overline{M}_{g,n}(E,d)$ is the moduli space of stable degree $d$ maps from genus $g$, $n$-pointed curves to $E$, and $ev_i$ is the evaluation map at the $i$th marked point. It's proved in \cite{Hurwitz} that (\ref{GW partition}) is a quasi-modular form in $\tau$ of weight $\sum\limits_{i=1}^n (k_i+2)$  under the identification $q=\exp(2\pi i \tau)$. \\

  In the B-model, let $\F^{\Etau}[L]=\sum\limits_{g\geq 0}\hbar^g \F^{\Etau}_g[L]$ be the effective functional on the polyvector fields $\E_\Etau=\PV^{*,*}_{\Etau}[[t]][2]$ of the elliptic curve $\Etau$ constructed in the previous section. As explained in definition \ref{higher-genus-B-invariants}, the higher genus quantum invariants are obtained via the limit $L\to \infty$
  $$
  \F^{\Etau}_g[\infty]:  \Sym^n\bracket{\HH^*(\Etau, \wedge^* T_{\Etau})[[t]]} \to \C
$$
Let $w$ be the linear coordinate on $\C$. We consider the following polyvector fields
\begin{align*}
   \omega= {i\over 2\im\ \tau}\pa_w\wedge d\bar w
\end{align*}
which is normalized such that $\Tr\ \omega\equiv\int_{\Etau} (\omega\vee dw)\wedge dw=1$. We consider
\begin{align}\label{Fg}
      \F^{\Etau}_g[\infty][t^{k_1}\omega,\cdots, t^{k_n}\omega]
\end{align}

We will prove that it is an almost holomorphic modular form of weight $\sum\limits_{i=1}^n (k_i+2)$. Therefore the following limit makes sense \cite{almost-modular-form}
\begin{align*}
   \lim\limits_{\bar\tau\to \infty}F^{\Etau}_g[\infty][t^{k_1}\omega,\cdots, t^{k_n}\omega]
\end{align*}
which gives a quasi-modular form with the same weight. The main theorem in this section is the following
\begin{thm}\label{main theorem} For any genus $g\geq 2$, $n>0$, and non-negative integers $k_1, \cdots, k_n$, we have the identity
\begin{eqnarray}
\sum_{d\geq 0}q^d\left\langle \prod_{i=1}^n\tau_{k_i}(\tilde \omega) \right\rangle_{g,d}=\lim\limits_{\bar\tau\to \infty}F^{\Etau}_g[\infty][t^{k_1}\omega,\cdots, t^{k_n}\omega]
\end{eqnarray}
under the identification $q=\exp(2\pi i\tau)$.
\end{thm}

It should be noted that the mysterious $\bar\tau\to\infty$ limit appears in \cite{BCOV} to describe the holomorphic anomaly and the large radius limit behavior of the topological string amplitudes. It's argued by physics method \cite{BCOV} that the quantum invariants constructed from Kodaira-Spencer gauge theory on Calabi-Yau manifolds can be identified with the Gromov-Witten invariants of its mirror Calabi-Yau under such limit. In general, the holomorphic anomaly arises from the choice a metric in Lemma \ref{isomorphism-polyvectors} to split the Hodge filtration, and the $\bar\tau\to \infty$ limits refers to taking the monodromy splitting filtration around large complex limit. This is discussed in detail in \cite{Si-Kevin}.

In our example of elliptic curves, the $\bar\tau\to\infty$ limit simply intertwines between the almost holomorphic modular forms and quasi-modular forms. This has also been observed in \cite{string-modular-form} in the study of local mirror symmetry.

In Theorem \ref{main theorem}, we only consider the input from $\HH^2(\Etau, \C)$ and its descendants, which is called stationary sector in \cite{Hurwitz}. In fact, descendant Gromov-Witten invariants with arbitrary inputs on $E$ can be obtained from the stationary sector via Virasoro equations \cite{virasoro}. Since we have proved that the same Virasoro equations hold for BCOV theory (see Theorem \ref{BCOV-virasoro-equations}), it follows from Theorem \ref{main theorem} that mirror symmetry actually holds for arbitrary inputs.

The rest of this section is devoted to prove Theorem \ref{main theorem}. We outline the structure as follows. In section \ref{mirror-propagator}, we analyze the BCOV propagator and give several equivalent descriptions that will be used. In section \ref{mirror-boson-fermion}, we briefly review the Boson-Fermion correspondence in the theory of lattice vertex algebra. In section \ref{mirror-Gromov-Witten}, we use Boson-Fermion correspondence to show that the partition function (\ref{GW partition}), computed in \cite{Hurwitz}, can be written as Feynman graph integrals with the BCOV propagator. In section \ref{mirror-limit}, we prove that (\ref{Fg}) is an almost holomorphic modular form and analyze the $\bar\tau\to \infty$ limit. In section \ref{mirror-proof}, we prove Theorem \ref{main theorem}.

\subsection{BCOV propagator on the elliptic curve}\label{mirror-propagator}
Let $\Etau=\mathbb C/\Lambda$ be the elliptic curve where $\Lambda=\mathbb Z+\mathbb Z \tau$, $\tau$ lies in the upper half plane. We will use the following convention for coordinates: let $w$ be the linear coordinate on $\mathbb C$, and the elliptic curve $\Etau$ is obtained via the identification $w\sim w+1, w\sim w+\tau$. We will denote by
\begin{align*}
q=e^{2\pi i\tau}
\end{align*}
and also use the $\C^*$ coordinate
\begin{align*}
    z=\exp(2\pi i w)
\end{align*}
such that $z\sim z q$ on the elliptic curve. We choose the standard flat metric on $\Etau$, and let $\Box$ be the Laplacian. The BCOV propagator is given by the kernel $\PP_\epsilon^L=\int_\epsilon^L du \bar\pa^*\pa e^{-u\Box}$, which is concentrated on $\PV^{0,0}_{\Etau}$ component. We normalize the integral such that $\PP_\epsilon^L$ is represented by
\begin{align*}
       \PP_\epsilon^L(w_1,w_2;\tau,\bar\tau)=-{1\over \pi}\int_\epsilon^L {du\over 4\pi u} \sum_{\lambda\in \Gamma} \left( {\bar w_{12}-\bar \lambda\over 4 u} \right)^2 e^{-{|w_{12}-\lambda|^2/4u}}
\end{align*}
where $w_{12}=w_1-w_2$. Note that it differs from the standard kernel by a factor ${{1\over \pi}}$. This factor is purely conventional and this choice will be convenient for the later discussion. Let $E_2(\tau)$ be the second Eisenstain series which is a quasi-modular form of weight 2
\begin{eqnarray*}
  E_2(\tau)={3\over \pi^2}\sum\limits_{n\in \Z}\sum^\prime\limits_{m\in \Z}{1\over (m+n\tau)^2}
  =1-24\sum_{n=1}^\infty {nq^n\over 1-q^n}
\end{eqnarray*}
where the sign $\sum\limits^\prime$ indicates that $(m,n)$ run through all $m\in \Z, n\in \Z$ with $(m,n)\neq (0,0)$. $E_2^*(\tau,\bar\tau)$ is the almost holomorphic modular form defined by
$$
E_2^*(\tau,\bar\tau)=E_2(\tau)-{3\over \pi \Im \tau}
$$

Note that $E_2(\tau)$ can be recovered from $E_2^*(\tau, \bar\tau)$ by taking the limit $\bar\tau\to \infty$ in the obvious sense.

\begin{lem}
 Under the limit $\epsilon\to 0, L\to \infty$, we have
 \begin{align}\label{limit propagator}
 \PP_0^\infty(w_1,w_2;\tau,\bar\tau)=-{1\over 4\pi^2}\wp(w_1-w_2;\tau)-{1\over 12}E_2^*(\tau,\bar\tau)
 \end{align}
 if $w_1-w_2\notin \Lambda$. Here $\wp(w;\tau)$ is Weierstrass's elliptic function
$$
    \wp(w;\tau)={1\over w^2}+\sum_{\lambda\in \Lambda, \lambda\neq
    0}\left({1\over (w-\lambda)^2}-{1\over \lambda^2} \right)
$$
\end{lem}
\begin{proof} This is a well-known result. See for example \cite{Li-modular} for an elementary proof.
\end{proof}
We will use the following notation to represent the $\bar\tau\to\infty$ limit, which we simply throw away the term involving $1\over \Im \tau$
\begin{eqnarray*}
      \PP_0^\infty(w_1,w_2;\tau,\infty)&\equiv&\lim_{\bar\tau\to\infty} \PP_0^\infty(w_1,w_2;\tau,\bar\tau)\nonumber\\&=&-{1\over 4\pi^2}\wp(w_1-w_2;\tau)-{1\over 12}E_2(\tau)\nonumber\\
      &=&-{1\over 4\pi^2}\sum_{n\in \Z}\sum_{m\in \Z}{1\over (w_1-w_2-(m+n\tau))^2}\label{limit of propagator}
\end{eqnarray*}
or simply $\PP_0^\infty(\tau,\infty)$ if no explicit coordinates are needed. We can also go to the $\C^*$-coordinate $z$ using the formula
$$
       \sum_{m\in \mathbb Z} {1\over (w+m)^2}=-4\pi^2 {z\over (1-z)^2}, \ \ z=\exp\left(2\pi i w\right)
$$
hence
$$
  \PP_0^\infty(w_1,w_2;\tau,\infty)=\sum_{n\in \Z}{z_1z_2q^n\over (z_1-z_2q^n)^2}, \ \ z_k=\exp\left(2\pi i w_k\right), k=1,2
$$

If we further assume that $w_1, w_2$ takes values in $\{a+b\tau|0\leq a,b<1\}$, then we have the following relation
$$
     |qz_2|<|z_1|<|q^{-1}z_2|
$$
and we get the power series expression
\begin{align}\label{BCOV propagator}\ \ \ \ \ \
\PP_0^\infty(w_1,w_2;\tau,\infty)={z_1z_2\over (z_1-z_2)^2}+\sum_{m\geq 1}{mz_1^mz_2^{-m}q^m\over 1-q^m}+
\sum_{m\geq 1}{m z_1^{-m}z_2^mq^m\over 1-q^m}
\end{align}

Later we will use this formula to give the Feynman diagram interpretation of the Gromov-Witten invariants on the elliptic curve.

\subsection{Boson-Fermion correspondence}\label{mirror-boson-fermion}
In this section, we discuss some examples of vertex algebra as well as their representations. We collect the basic results on Boson-Fermion correspondence that will be used to prove mirror symmetry. For more details, see \cite{vertex-algebra}\cite{soliton}.

\subsubsection{Free bosons}

The system of free boson is described by the infinite dimensional Lie algebra with basis $\{\alpha_n\}_{n\in \mathbb Z}$ and the commutator relations
\begin{align*}
     [\alpha_n,\alpha_m]=n \delta_{n+m,0}, \ \ n,m\in \Z
\end{align*}

The irreducible representations $\{\HH^B_p\}_{p\in \R}$ are indexed by the real number $p$ called ``momentum''. For each $\HH^B_p$, there exists an element $|p\rangle\in \HH^B_p$, which we call ``vacuum'', satisfying
\begin{align*}
      \alpha_0 |p\rangle =p |p\rangle , \ \alpha_{n}|p\rangle =0, \ n>0
\end{align*}
and the whole Fock space $\HH^B_p$ is given by
\begin{align*}
       \HH^B_p=\text{linear span of}\left\{\left.\alpha_{-i_1}^{k_1}\alpha_{-i_2}^{k_2}\cdots \alpha_{-i_n}^{k_n}|p\rangle \right| \ i_1>i_2>\cdots i_n>0, \ k_1,\cdots, k_n \geq 0, n\geq 0 \right\}
\end{align*}

We will be interested in the Fock space with zero momentum, where the vacuum vector is also annihilated  by $\alpha_0$
$$
   \alpha_n |0\rangle =0, \ \ \forall n\geq 0
$$
$\{\alpha_{-n}\}_{n>0}$ are called \emph{creation operators}, and $\{\alpha_{n}\}_{n>0}$ are called \emph{annihilation operators}.  We define the normal ordering $: :_B$ by putting all the annilation operators to the right,
\begin{align*}
     :\alpha_n\alpha_m:_B=\begin{cases}
        \alpha_n\alpha_m & \text{if}\ n\leq 0\\
       \alpha_m \alpha_n & \text{if}\ n>0
\end{cases}
\end{align*}
and similarly for the case with more $\alpha$'s. Here the subscript $``B"$ denotes the bosons in order to distinguish with the fermionic normal ordering that will be discussed later. It's useful to collect $\alpha_n$'s to form the following \emph{field}
\begin{align*}
   \alpha(z)=\sum_{n\in\mathbb Z}\alpha_n z^{-n-1}
\end{align*}
then we have the following relation
\begin{align*}
    \alpha(z)\alpha(w)=\sum_{n\geq 1}n z^{-n-1}w^{n-1}+:\alpha(z)\alpha(w):_B={1\over (z-w)^2}+ :\alpha(z)\alpha(w):_B, \ \ \text{if}\ |z|>|w|
\end{align*}
This provides a convenient way to organize the data of the operators and the normal ordering relations. We can construct the Virasoro operators acting on $\HH^B_p$ via normal ordering
\begin{align*}
   L_n={1\over 2}\sum_{i\in \mathbb Z}:\alpha_i \alpha_{n-i}:_B
\end{align*}
which satisfies the Virasoro algebra with central charge 1
\begin{align*}
   [L_m, L_n]=(m-n)L_{m+n}+{m^3-m\over 12}\delta_{m+n,0}, \ \ \forall n, m\in \Z
\end{align*}

If we consider the corresponding field
\begin{align*}
   L(z)=\sum_{n} L_n z^{-n-2}
\end{align*}
then we can write
\begin{align*}
     L(z)={1\over 2}:\alpha(z)^2:_B
\end{align*}
$L_0$ is called the ``energy operator'' and has the following expression
$$
    L_0={1\over 2}\alpha_0^2+\sum_{n\geq 1}\alpha_{-n} \alpha_n
$$
which acts on basis of $\HH^B_p$ as
$$
L_0  \alpha_{-i_1}^{k_1}\alpha_{-i_2}^{k_2}\cdots \alpha_{-i_n}^{k_n}|p\rangle=\left({1\over 2}p^2+\sum_{a=1}^n k_a i_a\right)  \alpha_{-i_1}^{k_1}\alpha_{-i_2}^{k_2}\cdots \alpha_{-i_n}^{k_n}|p\rangle
$$

The dual space $\HH^{B*}_p$ can be constructed similarly from the dual vacuum element $\langle p|\in \HH^{B*}_p$ such that
\begin{align*}
    \langle p| \alpha_0= p\langle p|, \ \langle p| \alpha_{-n}=0, n>0
\end{align*}
and
$$
 \HH^{B*}_p=\text{linear span of}\left\{\left.\langle p|\alpha_{i_n}^{k_n}\cdots\alpha_{i_2}^{k_2}\alpha_{i_1}^{k_1} \right| i_1>i_2>\cdots i_n>0, \ k_1,\cdots, k_n \geq 0, n\geq 0 \right\}
$$

The natural pairing
$$
\HH^{B*}_p\otimes  \HH^{B}_p \to \R
$$
is given by
$$
\langle p|\alpha_{j_m}^{l_m}\cdots\alpha_{j_2}^{l_2}\alpha_{j_1}^{l_1} \otimes \alpha_{-i_1}^{k_1}\alpha_{-i_2}^{k_2}\cdots \alpha_{-i_n}^{k_n}|p\rangle \to \langle p|\alpha_{j_m}^{l_m}\cdots\alpha_{j_2}^{l_2}\alpha_{j_1}^{l_1} \alpha_{-i_1}^{k_1}\alpha_{-i_2}^{k_2}\cdots \alpha_{-i_n}^{k_n}|p\rangle
$$
with the normalization condition
$$
     \langle p| |p\rangle=1
$$

There is a natural identification of the bosonic Fock space of integral momentum with
polynomial algebra $\mathbb C[z,z^{-1},x_1,x_2,\cdots]$ as follows. Let
$$
   H(x)=\exp\left(\sum_{n=1}^\infty x_n \alpha_n\right)
$$
then
$$
   \alpha_{-i_1}^{k_1}\alpha_{-i_2}^{k_2}\cdots \alpha_{-i_n}^{k_n}|m\rangle \to
   \sum_{l\in \mathbb Z}z^l\langle l|e^{H(x)}\alpha_{-i_1}^{k_1}\alpha_{-i_2}^{k_2}\cdots \alpha_{-i_n}^{k_n}|m\rangle \in \mathbb C[z,z^{-1},x_1,x_2,\cdots], \ \ m\in \Z
$$

Under this isomorphism the bosonic operators are represented by
$$
    \alpha_n \to {\partial\over \partial x_n}, \ \alpha_{-n}\to n x_n, \ \ n\geq 1
$$
and
$$
    \alpha_0\to z {\partial\over \partial z}
$$

\subsubsection{Free Fermions}
We consider the free fermionic system that is described by the infinite dimensional Lie superalgebra with odd basis $\{b_n\}, \{c_n\}$, indexed by $n\in \mathbb Z+1/2$, and the anti-commutator relations
\begin{align*}
   \{b_n, c_m\}=\delta_{m+n,0}, \ \{b_n, b_m\}=\{c_n,c_m\}=0, \ \ \forall n, m\in \Z+1/2
\end{align*}

The irreducible representation is given by the Fermionic Fock space $\HH^F$, which contains the vacuum $|0\rangle$ satisfying
\begin{align*}
   b_n |0\rangle =c_n |0\rangle =0, \ \forall n\in \mathbb Z^{\geq 0}+1/2
\end{align*}
and $\HH^F$ is constructed by
\begin{eqnarray*}
  \HH^F&=&\text{linear span of} \\ && \left\{\left.b_{-i_1}\cdots b_{-i_s}c_{-j_1}\cdots c_{-j_t}|0\rangle\right|
     0<i_1<i_2<\cdots<i_s, , 0<j_1<j_2<\cdots<j_t, \ s, t\geq 0\right \}
\end{eqnarray*}

The normal ordering $:\ :_F$ is defined similarly with extra care about the signs
\begin{align*}
    :b_n c_m:_F=\begin{cases}
                        b_n c_m & \text{if}\ n<0\\
                        -c_m b_n & \text{if}\ n>0
                       \end{cases}
\end{align*}
where the subscript $``F"$ refers to the fermions. We can also construct the Virasoro operators acting on $\HH^F$ via
\begin{align*}
   L_{n}={1\over 2}\sum_{k+l=n}(l-k):b_kc_l:_F=\sum_{k\in \mathbb Z+1/2}(n/2-k) :b_k c_{n-k}:_F, \ \ n\in \Z
\end{align*}
which satisfies the Virasoro algebra with central charge 1
\begin{align*}
   [L_m, L_n]=(m-n)L_{m+n}+{m^3-m\over 12}\delta_{m+n,0} \ \ \forall n,m\in \Z
\end{align*}

Similar to the bosonic case, we can collect the fermionic operators to form the fermionic fields
\begin{align*}
   b(z)=\sum_{n\in \mathbb Z+1/2}b_n z^{-n-1/2}, \ c(z)=\sum_{n\in \mathbb Z+1/2}c_n z^{-n-1/2}
\end{align*}
such that the normal ordering relations can be written in the simple form
\begin{align*}
   b(z)c(w)={1\over z-w}+:b(z)c(w):_F, \ \ \text{if}\ |z|>|w|
\end{align*}

The Virasoro operators can be collected
$$
      L(z)=\sum_{n\in Z}L_n z^{-n-2}
$$
and it's easy to see that
\begin{align*}
   L(z)={1\over 2}:\partial b(z) c(z):_F-{1\over 2}:b(z)\partial c(z):_F
\end{align*}

The energy operator $L_0$ has the expression
$$
   L_0=\sum_{k\in \mathbb Z^{\geq 0}+1/2}k (b_{-k}c_k+c_{-k}b_k)
$$

\subsubsection{From fermions to bosons}

Consider the above free fermionic system with fields $b(z), c(z)$. We construct the following bosonic field
\begin{align*}
    \alpha(z)=:b(z)c(z):_F
\end{align*}

In mode expansions,
\begin{align*}
     \alpha_{n}=\sum_{k\in\mathbb Z+1/2}:b_k c_{n-k}:_F
\end{align*}

It's easy to see that the following commutator relations hold as operators on $\HH^F$
\begin{equation*}
\begin{split}
            & [\alpha_m,\alpha_n]=m\delta_{m+n,0}\\
            & [\alpha_m, b_n]= b_{m+n}\\
            & [\alpha_m, c_n]=-c_{m+n}
\end{split}
\end{equation*}
Therefore $\alpha(z)$ defines a free bosonic field. Moreover, the Virasoro operators coincide for bosons and fermions, i.e.
\begin{align*}
     L(z)={1\over 2}:\alpha(z)^2:_B={1\over 2}:\partial b(z) c(z):_F-{1\over 2}:b(z)\partial c(z):_F
\end{align*}

Consider the charge operator $\alpha_0$, which corresponds to bosonic momentum operator
$$
    \alpha_0=\sum_{k\in\mathbb Z^{\geq 0}+1/2} (b_{-k} c_{k}-c_{-k}b_k )
$$
$\alpha_0$ acts on the basis of the Fock space as
$$
   \alpha_0 b_{-i_1}\cdots b_{-i_s}c_{-j_1}\cdots c_{-j_t}|0\rangle=\left(s-t\right) b_{-i_1}\cdots b_{-i_s}c_{-j_1}\cdots c_{-j_t}|0\rangle
$$
$\HH^F$ is decomposed into eigenvectors of $\alpha_0$
\begin{align*}
   \HH^F=\bigoplus\limits_{m\in \mathbb Z}\HH^F_m
\end{align*}
such that each $\HH^F_m$ gives a representation of the free bosons. For each $\HH^F_m$, there's a special element given by
\begin{align*}
    |m\rangle=\begin{cases}
                     |0\rangle & \text{if}\ m=0\\
                     b_{-m+{1/2}}\cdots b_{-{1/2}}|0\rangle & \text{if}\ m>0\\
                     c_{ m+{1/2}}\cdots c_{-{1/2}}|0\rangle & \text{if}\ m<0\\
                     \end{cases}
\end{align*}

It's easy to see that
$$
     \alpha_n |m\rangle=0, \ \forall n\in \mathbb Z^{>0}, \forall m\in \Z
$$

\begin{prop} The representation $\HH^F_m$ of free bosons is isomorphic to the Fock space $\HH^B_m$ with momentum $m\in \Z$ under the identification of vacuums
$$
|m\rangle\Leftrightarrow\begin{cases}
                     |0\rangle & \text{if}\ m=0\\
                     b_{-m+{1/2}}\cdots b_{-{1/2}}|0\rangle & \text{if}\ m>0\\
                     c_{ m+{1/2}}\cdots c_{-{1/2}}|0\rangle & \text{if}\ m<0\\
                     \end{cases}
$$
\end{prop}

\subsubsection{From bosons to fermions}

Let $P$ be the creation operator for momemtum on bosonic fock space defined by
$$
    e^P|m\rangle=|m+1\rangle
$$

It follows that we have the following commutator relation
$$
    [\alpha_0, P]=1
$$

We define formally
\begin{align*}
    \phi(z)=P+\alpha_0 \log z +\sum_{n\neq 0} {\alpha_{-n}\over n}z^n
\end{align*}
where $\alpha(z)$ is related to $\phi(z)$ by
$$
     \alpha(z)=\partial_z\phi(z)
$$

Since $\alpha_0|0\rangle=0$, we view $\alpha_0$ as annihilation operator and $P$ as creation operator, and extend the bosonic normal ordering by
$$
    :\alpha_0P:_B=:P\alpha_0:_B=P\alpha_0
$$

Direct calculation shows
$$
    \phi(z)\phi(w)=\ln (z-w)+:\phi(z)\phi(w):_B, \ \ \text{if}\ |w|<|z|
$$
\begin{prop}Under the above identification of fermionic Fock space $\HH^F$ with bosonic Fock space $\bigoplus\limits_{m\in \Z}\HH^B_m$, the fermionic fields can be represented by bosonic fields acting on $\bigoplus\limits_{m\in \Z}\HH^B_m$ as
\begin{eqnarray}
    b(z)=:e^{\phi(z)}:_B, \ \ c(z)=:e^{-\phi(z)}:_B
\end{eqnarray}
\end{prop}

As an example, we can put the product of two fermionic fields into normal ordered form in two ways. Within fermionic fields
$$
   b(z)c(w)={1\over z-w}+:b(z)c(w):_F
$$
or using the bosonic representation
$$
  b(z)c(w)=:e^{\phi(z)}:_B:e^{-\phi(w)}:_B={1\over z-w}:e^{\phi(z)-\phi(w)}:_B
$$
where in the second equality we have used the Wick's theorem (see for example \cite{soliton}). Therefore
\begin{eqnarray}\label{bosonization}
   :b(z)c(w):_F={1\over z-w}\left(:e^{\phi(z)-\phi(w)}:_B-1\right)
\end{eqnarray}

See \cite{soliton} for a more systematic treatment of the above formula.

\subsection{Gromov-Witten invariants on elliptic curves}\label{mirror-Gromov-Witten}
\subsubsection{Stationary Gromov-Witten invariants}

Let $E$ be an elliptic curve. The Gromov-Witten theory on $E$ concerns the moduli space
$$
       \overline{M}_{g,n}(E,d)
$$
parametrizing connected, genus $g$, n-pointed stable maps to $E$ of degree $d$. Let
$$
    ev_i: \overline{M}_{g,n}(E,d)\to E
$$
be the morphism defined by evaluation at the $i$th marked point. Let $\tilde \omega$ denote the Poincar\'{e} dual of the point class, $\psi_i\in H^2(\overline{M}_{g,n}(E,d), \Q)$ the first Chern class of the cotangent line bundle $L_i$ on the moduli space $\overline{M}_{g,n}(E,d)$. By the Virasoro constraints proved in \cite{virasoro}, the full descendant Gromov-Witten invariants on $E$ are determined by the \emph{stationary sector}, i.e.,
\begin{align*}
    \left\langle \prod\limits_{i=1}^n\tau_{k_i}\tilde \omega \right\rangle_{g,d}=\int_{\left[\overline{M}_{g,n}(E,d)\right]^{vir}} \prod\limits_{i=1}^n \psi_i^{k_i} ev_i^*(\tilde \omega)
\end{align*}
where $\left[\overline{M}_{g,n}(E,d)\right]^{vir}$ is the virtual fundamental class of $\overline{M}_{g,n}(E,d)$. The integral vanishes unless the dimension constraint
\begin{align*}
    \sum\limits_{i=1}^n k_i=2g-2
\end{align*}
is satisfied. Therefore we can omit the subscript $g$ in the bracket $\langle\ \rangle$. We can also consider the disconnected theory as in \cite{Hurwitz}, where the domain curve of the stable map is allowed to have disconnected components. The bracket $\langle\ \rangle^{dis}$ will be used for the disconnected Gromov-Witten invariants. It's proved in \cite{Hurwitz} that the stationary Gromov-Witten invariants can be computed through fermionic vertex algebra, which we now describe. \\

Let $\HH^F$ be the Fock space of free fermionic algebra with fermionic fields $b(z), c(z)$, $\HH^F_0$ is the subspace annihilated by the charge operator $\alpha_0$. Consider the following operator
\begin{align*}
   \mc E(z; \lambda)=\sum_{n\in\mathbb Z} \mc E_n(\lambda) z^{-n-1}=:b(e^{\lambda/2}z) c(e^{-\lambda/2}z):_F+{1\over (e^{\lambda/2}-e^{-\lambda/2})z}
\end{align*}
In components, we can formally write
\begin{align*}
        \mc E_n(\lambda)=\oint dz z^n\mc E(z; \lambda)= \begin{cases}
        \sum\limits_{k\in \Z+{1\over 2}}e^{\lambda k}:b_{-k}c_k:+{1\over (e^{\lambda/2}-e^{-\lambda/2})} & \text{if}\ n=0\\
        \sum\limits_{k\in \Z+{1\over 2}}b_{n-k}c_k e^{\lambda(k-n/2)}&\text{if}\  n\neq 0\\
        \end{cases}
\end{align*}
Here $\oint={1\over 2\pi i}\int_C$, where $C$ is a circle surrounding the origin. Decomposing in terms of powers of $\lambda$, we define
\begin{align*}
     \mc E(z; \lambda)=\sum_{n\geq -1} {\lambda^{n}}\mc E^{(n)}(z)
\end{align*}

Consider the following n-point partition function for the stationary GW invariants of the elliptic curve $E$:
\begin{align*}
   F_E(\lambda_1,\cdots,\lambda_n;q)=\sum_{d\geq 0}q^d \left\langle \prod_{i=1}^n \left(\sum_{k\geq -2}\lambda_i^{k}\tau_{k}(\tilde \omega)\right) \right\rangle_d^{dis}
\end{align*}

The bracket is the disconnected descendant GW invariants.

\begin{prop}[\cite{Hurwitz}]\label{stationary GW} The above partition function can be written as a trace on the fermionic Fock space
\begin{align}
%     \sum_{d\geq 0}q^d \left\langle
%\exp\left( \sum_{k\geq 0}t_k \tau_{k}(\tilde \omega) \right)\right\rangle_d^{dis} =\Tr_{F_0} %q^{L_0} \exp\left( \sum_{k\geq 0}t_k \oint dz \mc  E^{(k+1)}(z)  \right)
\sum_{d\geq 0}q^d \left\langle \prod_{i=1}^n \left(\sum_{k\geq -2}\lambda_i^{k}\tau_{k}(\tilde \omega)\right) \right\rangle_d^{dis}=\Tr_{\HH^F_0} q^{L_0} \prod_{i=1}^n  {1\over \lambda_i}\oint dz \mc E(z;\lambda_i)
\end{align}
where we use the convention as in \emph{\cite{Hurwitz}}
$$
   \tau_{-2}(\tilde \omega)=1, \ \ \tau_{-1}(\tilde \omega)=0
$$
\end{prop}
In \cite{Hurwitz}, the fermionic Fock space is represented by the infinite wedge space $\Lambda^{\infty\over 2}V$, where $V$ is a linear space with basis $\underline{k}$ indexed by the half-integers:
$$
       V=\bigoplus\limits_{k\in \Z+{1\over 2}} \C \underline{k}
$$

For reader's convenience, the notations used in \cite{Hurwitz} are related to our notations here via
$$
       \psi_k\to b_{-k}, \ \ \psi^*_k\to c_k, \ \ C\to \alpha_0, \ \ H\to L_0
$$

\subsubsection{Bosonization}
Using fermion-boson correspondence, we can have a bosonic description of the Gromov-Witten invariants on the elliptic curve. Following  the bosonization rule
$$
    b(z)=:e^{\phi(z)}:_B, \ \ c(z)=:e^{-\phi(z)}:_B, \ \ \ \mbox{where}\ \phi(z)=P+\alpha_0 \log z +\sum_{n\neq 0} {\alpha_{-n}\over n}z^n
$$
where $P$ is the creation operator for momentum. Using Eqn (\ref{bosonization}), we can write $\mc E(z;\lambda)$ in terms of bosonic fields
$$
   \mc E(z; \lambda)={1\over (e^{\lambda/2}-e^{-\lambda/2})z} :e^{\phi(e^{\lambda/2}z)-\phi(e^{-\lambda/2}z)}:_B
$$

Let $\mc S(t)$ be the function
$$
      \mc S(t)={e^{t/2}-e^{-t/2}\over t}={\sinh(t/2)\over t/2}
$$
Then
$$
   \mc E(z; \lambda)={1\over \lambda \mc S(\lambda) z}:\exp\left(
     \mc S(\lambda z\partial_z)(\lambda z\alpha(z)) \right) :_B, \ \ \ \alpha(z)=\partial_z\phi(z)
$$

The following lemma on the interpretation of the factor ${1\over S(\lambda)}$ will be used later in the Feynman diagram representation of the descendant Gromov-Witten invariants.
\begin{lem}\label{self-loop}
\begin{eqnarray*}         {1\over \mc S(\lambda)}=\exp\left( {\lambda^2\over 2}\left.\mc S\left({1\over 2\pi i}\lambda \partial_{w_1}\right)
\mc S\left({1\over 2\pi i}\lambda \partial_{w_2}\right)\left(\sum_{n\in \mathbb Z\backslash \{0\}}{1\over (2\pi i)^2}{1\over (w_1-w_2+n)^2} \right)\right|_{w_1=w_2} \right)
\end{eqnarray*}
\end{lem}
\begin{proof} Since $\mc S(t)$ is an even function of $t$, we have
\begin{eqnarray*}
&&{\lambda^2\over 2}\left.\mc S\left({1\over 2\pi i}\lambda \partial_{w_1}\right)
\mc S\left({1\over 2\pi i}\lambda \partial_{w_2}\right)\left(\sum_{n\in \mathbb Z\backslash \{0\}}{1\over (2\pi i)^2}{1\over (w_1-w_2+n)^2} \right)\right|_{w_1=w_2}
\\ &=&{ (\lambda/2\pi i)^2}\left. \mc S((\lambda/2\pi i)\partial_w)^2
\left(\sum_{n\geq 1}{1\over (w+n)^2} \right) \right|_{w=0} \\
&=&\left.\sum_{k\geq 1}{2(\lambda/2\pi i)^{2k}\over (2k)!}\left({\partial\over\partial w}\right)^{2k-2}
\left(\sum_{n\geq 1}{1\over (w+n)^2} \right)\right|_{w=0}
\\&=&\sum_{k\geq 1}{(\lambda/2\pi i)^{2k}\over k}\sum_{n\geq 1}{1\over n^{2k}}
\\&=&\sum_{n\geq 1}\sum_{k\geq 1}{(\lambda/2\pi i n)^{2k}\over k}
\\&=&-\sum_{n\geq 1}\ln \left(1+{\lambda^2\over (2\pi)^2n^2}\right)
\end{eqnarray*}

On the other hand, from the formula ${\sin \lambda\over \lambda}=\prod_{\geq 1}\left(1-{\lambda^2\over n^2\pi^2}\right)$, we see that
$$
    { S(\lambda)}= { \sinh \lambda/2\over \lambda/2}=\prod_{n\geq 1}\left(1+{\lambda^2\over n^2(2\pi)^2} \right)
$$
this proves the lemma.
\end{proof}
%The above calculation has the following interpretation. In the vertex algebra calculation, we have the %OPE
%$$
%    z_1\alpha(z_1) z_2\alpha(z_2)\sim {z_1z_2\over (z_1-z_2)^2}={1\over (2\pi i)^2}\sum_{n\in \mathbb Z}{1\over (w_1-w_2+n)^2}, \ \ \ z_i=\exp (2\pi i w_i)
%$$
%The normal ordering $::_B$ essentially disgards the above singular term. However, in the BCOV calculation, we only disgard the term ${1\over (w_1-w_2)^2}$ for self-loops. The discrepancy is given above which contributes the factor ${1\over \mc S(\lambda)}$.

\subsubsection{Feynman Diagram Representation}

Let $w$ be the $\C$ coordinate where the elliptic curve is defined via the equivalence: $w\sim w+1\sim w+\tau$. We identify $z$ with the coordinate on $\C^*$, such that
$$
       z=\exp(2\pi i w)
$$

Consider the following bosonic lagrangian on $\PV^{0,0}_{\Etau}$ coming from the above bosonization
\begin{align}\label{definition-lagrangian}
      \sum_{k\geq -1} \lambda^{k}\mc L^{(k)}(\mu(w))\equiv{1\over \lambda}\exp\left(
     \mc S\left({\lambda\over 2\pi i} \partial_w\right)(\lambda\mu(w)) \right), \ \ k\geq -1
\end{align}
where on the right hand side, we can expand the lagrangian in terms of powers of $\lambda$, which defines $\mc L^{(k)}$. Let $C$ be a representative of the homology class of the circle $[0,1]$ on the elliptic curve. Let $I_C^{(k)}$ be the functional on $\PV^{0,0}_{\Etau}$ given by
$$
       I_C^{(k)}[\mu]=  \int_C {dw} \mc L^{(k)}(\mu(w)), \ \ \mu\in \PV^{0,0}_{\Etau}
$$
\begin{prop}\label{stationary GW}
The stationary GW invariants can be represented by Feynman integrals
$$
 {\sum\limits_{d\geq 0}q^d \left\langle \prod_{i=1}^n\tau_{k_i}(\tilde \omega)\right\rangle_d^{dis}\over \sum\limits_{d\geq 0}q^d\left\langle1 \right\rangle^{dis}_d}=\lim_{\bar\tau\to\infty}\lim_{\substack{\epsilon\to 0\\ L\to \infty}}W^{dis}\left(\PP^{L}_{\epsilon}; I_{C_1}^{(k_i+1)},\cdots, I_{C_n}^{(k_n+1)}  \right)
$$
where the $C_i$'s are representatives of the homology class of the cycle $[0,1]$ and are chosen to be disjoint. $W^{dis}$ is given by the weighted summation of all Feynman diagrams (possibly disconnected) with
$n$ vertices $I_{C_1}^{(k_1+1)},\cdots, I_{C_n}^{(k_n+1)}$ and the propagator $\PP^{L}_{\epsilon}$ which is the BCOV propagator. The normalization factor on the LHS is
$$
  \sum_{d\geq 0}q^d\left\langle1 \right\rangle^{dis}_d=  {1\over \prod\limits_{i=1}^\infty (1-q^i)}
$$
\end{prop}
\begin{proof} We will use $w$'s for coordinates on $\C$ and $z$'s for coordinates on $\C^*$. We use the conventions that are used in section \ref{mirror-propagator}. The BCOV propagator can be written as
$$
       \PP_\epsilon^L(w_1,w_2;\tau,\bar\tau)=-{1\over \pi}\int_\epsilon^L {du\over 4\pi u} \sum_{\lambda\in \Gamma} \left( {\bar w_{12}-\bar \lambda\over 4 u} \right)^2 e^{-{|w_{12}-\lambda|^2/4u}}, \ \ \mathrm{where}\ \ w_{12}=w_1-w_2
$$
and under the limit $\epsilon\to 0, L\to \infty, \bar\tau\to\infty$,
$$
\PP_{0}^\infty(w_1,w_2;\tau,\infty)=\sum_{m\in\mathbb Z}{z_1z_2 q^m\over (z_1-z_2 q^m)^2}={z_1z_2\over (z_1-z_2)^2}+\sum_{m\geq 1}
    {m z_1^m z_2^{-m}q^{m}\over 1-q^m}+\sum_{m\geq 1}
    {m z_1^{-m} z_2^{m}q^{m}\over 1-q^m}
$$
where $z_i=\exp(2\pi i w_i)$ and $|q|\leq |z_i|<1$ for $i=1,2$. Let $\{C_i\}_{1\leq i\leq n}$ be disjoint cycles lying in the annulus $\{z\in \C^*||q|<|z|<1\}$ and representing the generator of the fundamental group of $\C^*$ , such that $C_{i}$ lies entirely outside $C_{i+1}$ for $1\leq i<n$. By Proposition \ref{stationary GW} and the boson-fermion correspondence
\begin{align}
   &\sum_{d\geq 0}q^d \left\langle \prod_{i=1}^n \left(\sum_{k\geq -2}\lambda_i^{k}\tau_{k}(\tilde \omega)\right) \right\rangle_d^{dis}\nonumber\\
   =&\Tr_{\HH^B_0} q^{L_0}\prod_{i=1}^n  {1\over \lambda_i^2}\oint_{C_i} {dz\over z} {1\over  \mc S(\lambda_i) }:\exp\left(
     \mc S(\lambda_i z\partial_z)(\lambda_i z\alpha(z)) \right) :_B \nonumber
  \\=& \sum_{k_1\geq 0,k_2\geq 0,\cdots}\prod_{i=1}^{\infty}{q^{ik_i}\over i^{k_i}k_i!}
  \left\langle0\left| \left(\prod_{i=1}^\infty \alpha_i^{k_i}\right)  \prod_{i=1}^n  {1\over \lambda_i^2}\oint_{C_i} {dz\over z} {1\over  \mc S(\lambda_i) }
:\exp\left(
     \mc S(\lambda_i z\partial_z)(\lambda_i z\alpha(z)) \right) :_B  \left(\prod_{i=1}^\infty \alpha_{-i}^{k_i}\right) \right| 0\right\rangle\label{bosonized GW}
\end{align}

Using Wick's Theorem (see for example \cite{soliton}), we can put the expression in the bracket into the normal ordered form, and the above summation can be expressed in terms of Feynman diagrams. From the normal ordering relations
\begin{eqnarray*}
       z_1\alpha(z_1)z_2\alpha(z_2)&=&{z_1z_2\over (z_1-z_2)^2}+:z_1\alpha(z_1)z_2\alpha(z_2):_B, \ \ |z_1|>|z_2|\\
      \alpha_n z\alpha(z)&=& n z^{n}+:\alpha_n z\alpha(z):_B \ \ n>0\\
      z\alpha(z)\alpha_{-n}&=& n z^{-n}+:z\alpha(z)\alpha_{-n}:_B\ \ n>0\\
      \alpha_n \alpha_{-n}&=&n+ :\alpha_n\alpha_{-n}:_B\ \ n>0
\end{eqnarray*}
we see that there're two types of vertices for the Feynman diagrams.
\begin{enumerate}
\item The Type I vertices are given by
$$
\exp\left(
     \mc S(\lambda_i z\partial_z)(\lambda_i z\alpha(z)) \right)
$$
for each $\lambda_i, 1\leq i\leq n$, where $z\alpha(z)$ is viewed as input.
\item The Type II are vertices of valency two for each $m>0$, with two inputs $\alpha_m, \alpha_{-m}$ and weight ${q^m\over m}$, i.e., vertices of the form
$$
    {q^m\over m}\alpha_m \alpha_{-m}, \ \ m>0
$$
This vertex comes in pairs from $ \left(\prod\limits_{i=1}^\infty \alpha_i^{k_i}\right) $ and $ \left(\prod\limits_{i=1}^\infty \alpha_{-i}^{k_i}\right)$.
\end{enumerate}
The propagators also have three types.
\begin{enumerate}
\item The Type A propagators connect $z_1\alpha(z_1)$ and $z_2\alpha(z_2)$ at two different vertices of Type I and gives the value
$$
       {z_1z_2\over (z_1-z_2)^2}
$$
\item The Type B propagators connects $z\alpha(z)$ from the vertices of Type I and $\alpha_m$ from the vertices of Type II, which gives the value
$$
           |m| z^{m}, \ \ m\in \Z\backslash \{0\}
$$
\item The Type C propagators connect $\alpha_m, \alpha_{-m}$ from two \emph{different} vertices of Type II, which gives the value
$$
          |m|
$$
\end{enumerate}
Since the vertex of Type II has valency two, we can insert any number of vertices of Type II into the propagator of Type A using propagator of Type C. This is equivalent to considering only vertices of Type I from $ \exp\left(
     \mc S(\lambda_i z\partial_z)(\lambda z\alpha(z)) \right) $ but with propagators
\begin{align}\label{propagator I}
       \left({z_1z_2\over (z_1-z_2)^2}+\sum_{m\geq 1}
    {m z_1^m z_2^{-m}q^{m}\over 1-q^m}+\sum_{m\geq 1}
    {m z_1^{-m} z_2^{m}q^{m}\over 1-q^m} \right)
\end{align}
connecting $z_1\alpha(z_1)$ and $z_2\alpha(z_2)$ at two different vertices, and
\begin{align}\label{propagator II}
      \left. \left(\sum_{m\geq 1}
    {m q^{m}z_1^mz_2^{-m}\over 1-q^m}+\sum_{m\geq 1}
    {m q^{m}z_1^{-m}z_2^m\over 1-q^m} \right)\right|_{z_1=z_2}
\end{align}
for propagator connecting two $z\alpha(z)$'s at the same vertex.

Now we compare it with the Feynman integral
$$
\sum_{k_i\geq -2}\lim_{\bar\tau\to\infty}\lim_{\substack{\epsilon\to 0\\ L\to \infty}}W^{dis}\left(\PP^{L}_{\epsilon}; \lambda_1^{k_i}I_{C_1}^{(k_i+1)},\cdots, \lambda_{n}^{k_n}I_{C_n}^{(k_n+1)}  \right)
$$
The vertices $I_{C_1}^{(k_i+1)}[\mu(w)]$ are precisely the same by construction via the identification of fields
$$
    \mu(w)=z\alpha(z)
$$

The propagator connecting two \emph{different} vertices $I_{C_1}^{(k_i+1)}$ and $I_{C_1}^{(k_j+1)}$ for $i\neq j$ is
$$
  \lim_{\bar\tau\to\infty}\lim_{\substack{\epsilon\to 0\\ L\to \infty}}\PP^{L}_{\epsilon}(w_1-w_2;\tau, \bar\tau)=  \left({z_1z_2\over (z_1-z_2)^2}+\sum_{m\geq 1}
    {m z_1^m z_2^{-m}q^{m}\over 1-q^m}+\sum_{m\geq 1}
    {m z_1^{-m} z_2^{m}q^{m}\over 1-q^m} \right)
$$
by equation (\ref{BCOV propagator}), where $z_i=\exp(2\pi i w_i), i=1,2$. This is precisely (\ref{propagator I}). To consider the self-loop contributions, note that the regularized BCOV propagator is given by the sum
\begin{eqnarray*}
       &&\PP_\epsilon^L(w_1,w_2;\tau,\bar\tau)\\&=&-{1\over \pi}\int_\epsilon^L {du\over 4\pi u} \sum_{\lambda\in \Gamma} \left( {\bar w_{12}-\bar \lambda\over 4 u} \right)^2 e^{-{|w_{12}-\lambda|^2/4u}}\\
       &=&-{1\over \pi}\int_\epsilon^L {du\over 4\pi u} \left( {\bar w_{1}-\bar w_2\over 4 u} \right)^2 e^{-{|w_{1}-w_2/4u}}-{1\over \pi}\int_\epsilon^L {du\over 4\pi u} \sum_{\lambda\in \Gamma, \lambda\neq 0} \left( {\bar w_{1}-\bar w_2-\bar \lambda\over 4 u} \right)^2 e^{-{|w_{1}-w_2-\lambda|^2/4u}}\\
\end{eqnarray*}
Since the vertices $I_{C_1}^{(k_i+1)}$ contains only holomorphic derivatives, the first term doesn't contribute to the self-loops after setting $w_1=w_2$, while the second is smooth around the diagonal $w_1=w_2$. By Eqn (\ref{limit propagator}), under the limit $\lim\limits_{\bar\tau\to\infty}\lim\limits_{\substack{\epsilon\to 0\\ L\to \infty}}$, the propagator for the self-loop is equivalent to
\begin{eqnarray*}
     &&-{1\over 4\pi^2}\sum_{m\in \Z\backslash\{0\}}{1\over (w_1-w_2-m)^2}   -{1\over 4\pi^2}\sum_{n\in \Z\backslash\{0\}}\sum_{m\in \Z}{1\over (w_1-w_2-(m+n\tau))^2}\\
     &=&-{1\over 4\pi^2}\sum_{m\in \Z\backslash\{0\}}{1\over (w_1-w_2-m)^2} +\left(\sum_{m\geq 1}
    {m q^{m}z_1^mz_2^{-m}\over 1-q^m}+\sum_{m\geq 1}
    {m q^{m}z_1^{-m}z_2^m\over 1-q^m} \right)
\end{eqnarray*}
which differs from (\ref{propagator II}) by the first term. By Lemma \ref{self-loop}, the first term contributes precisely the factor ${1\over \mc S(\lambda)}$ in (\ref{bosonized GW}). This proves the theorem.
\end{proof}
\begin{rmk} If all $k_i$'s are taken to be $1$, then it reduces to Dijkgraaf's theorem in \cite{Dijkgraaf}, where the RHS are given by cubic Feynman diagrams. Dijkgraaf proves that the corresponding cubic Feynman integrals compute certain Hurwitz numbers on the elliptic curve, which  can be identified with the stationary descendant Gromov-Witten invariants with input $\tau_1(\tilde w)$ under the Hurwitz/Gromov-Witten correspondence \cite{Hurwitz}.
\end{rmk}
 We can further decompose the lagrangian by the number of derivatives
$$
 \mc L^{(k)}(\mu)=\sum_{g\geq 0}\mc L_{g}^{(k)}(\mu)
$$
where $L_{g}^{(k)}(\mu)$ contains $2g$ derivatives. Let $I_{C,\hbar}^{(k)}$ be the functional on $\PV^{0,0}_{\Etau}$ taking value in $\C[[\hbar]]$ that is given by
$$
       I_{C,\hbar}^{(k)}[\mu]= \sum\limits_{g\geq 0}\hbar^g \int_C {dw} \mc L_g^{(k)}(\mu(w)), \ \ \mu\in \PV^{0,0}_{\Etau}
$$
where $C$ is a cycle representing the class $[0,1]$ as before.
\begin{cor}\label{hbar stationary GW}
With the same notations as in Proposition \ref{stationary GW}, we have
\begin{align}
   { {1\over \hbar}\sum\limits_{d\geq 0}q^d \hbar^{g}\left\langle \prod_{i=1}^n\tau_{k_i}(\tilde \omega)\right\rangle_{g,d}^{dis}\over \sum\limits_{d\geq 0}q^d \langle 1 \rangle^{dis}_d}=
\lim_{\bar\tau\to \infty}\lim_{\substack{\epsilon\to 0\\ L\to \infty}}\bracket{\exp\left({\hbar} {\pa\over \pa {\PP^{L}_{\epsilon}}} \right) \prod_{i=1}^n {1\over \hbar } I_{C_i, \hbar}^{(k_i)}}[0]
\end{align}
where on the right hand side, it's understood that the external inputs are zero.
\end{cor}
\begin{proof} Proposition \ref{stationary GW} can be rewritten as
\begin{align}
   {\sum\limits_{d\geq 0}\sum\limits_{g\geq 0}q^d\left\langle \prod\limits_{i=1}^n\sum\limits_{k_i\geq -2}\lambda_i^{k_i}\tau_{k_i}(\tilde \omega)\right\rangle_{g,d}^{dis}\over \sum\limits_{d}q^d \langle 1 \rangle^{dis}_d}\nonumber=
\lim\limits_{\bar\tau\to \infty}\lim_{\substack{\epsilon\to 0\\ L\to \infty}}\exp\left( {\pa\over \pa {\PP^{L}_{\epsilon}}} \right) \prod_{i=1}^n  \left({1\over \lambda_i}\sum\limits_{k_i\geq -1}\lambda_i^{k_i}I_{C_i}^{(k)} \right)[0]
\end{align}
The theorem follows easily from Eqn (\ref{definition-lagrangian}) and the rescaling of the above equation under
$$
    \lambda_i\to \lambda_i \sqrt{\hbar}
$$
\end{proof}
\subsection{Modularity and $\bar \tau\to \infty$ limit}\label{mirror-limit}
Kodaira-Spencer gauge theory is known to be the closed string field theory of B-twisted topological string. It's argued in \cite{BCOV} by string theory technique that the B-twisted topological string amplitude would have a meaningful $\bar t\to \infty$ limit around the large complex limit of the Calabi-Yau manifold. Here $t$ is certain coordinates on the moduli space of complex structures. We will investigate the meaning of $\bar\tau\to \infty$ for the elliptic curve example in this section.

\subsubsection{$\bar \tau\to \infty$ limit}

Let $\omega={i\over 2\Im \tau}{\pa_w} \wedge d\bar w\in \PV^{1,1}_{\Etau}$, which is normalized such that
$$
   \Tr\ \omega=1
$$

Let $\F^{\Etau}[L]=\sum\limits_{g\geq 0}\hbar^g\F^{\Etau}_g[L]$ be the family of effective action constructed from quantizing the BCOV theory on $\Etau$,  and
$$
      \F^{\Etau}_g[\infty]\equiv \lim\limits_{L\to \infty}\F_g^{\Etau}[L]
$$

We are interested in the following correlation functions
$$
   \F^{\Etau}_g[\infty][t^{k_1}\omega,\cdots,t^{k_n}\omega]
$$
for some non-negative integers $k_1,\cdots, k_n$ satisfying the Hodge weight condition
$$
   \sum\limits_{i=1}^n k_i=2g-2
$$

Let $\mathbb H=\{\tau\in \C |\Im \tau>0\}$ be the complex upper half-plane. The group $SL(2.\Z)$ acts on $\mc H$ by
$$
   \tau\to \gamma\tau={A\tau+B\over C\tau+D}, \ \ \mbox{for}\ \gamma\in \begin{pmatrix} A& B\\ C& D \end{pmatrix} \in SL(2, \Z)
$$
Recall that an \emph{almost holomorphic modular form} \cite{almost-modular-form} of weight $k$ on $SL(2,\Z)$ is a function
$$
       \hat f: \mathbb H\to \C
$$
which grows at most polynomially in $1/\Im (\tau)$ as $\Im(\tau)\to 0$ and
satisfies the transformation property
$$
          \hat f(\gamma \tau)=(C\tau+D)^k f(\tau) \ \ \mbox{for all}\ \gamma\in \begin{pmatrix} A& B\\ C& D \end{pmatrix} \in SL(2, \Z)
$$
and has the form
$$
   \hat f(\tau, \bar\tau)=\sum_{m=0}^Mf_m(\tau) \Im(\tau)^{-m}
$$
for some integer $M\geq 0$, where the functions $f_m(\tau)$'s are holomorphic in $\tau$. The following limit makes sense
$$
       \lim\limits_{\bar\tau\to \infty} \hat f(\tau, \bar\tau)= f_0(\tau)
$$
which gives the isomorphism between the rings of almost holomorphic modular forms and quasi-modular forms \cite{almost-modular-form}. A general construction of such modular forms are described in \cite{Li-modular} by graph integrals

\begin{lem}[\cite{Li-modular}]\label{almost modular forms}
Let $\Gamma$ be a connected oriented graph, $V(\Gamma)$ be the set of vertices, $E(\Gamma)$ be the set of edges, and $l,r: E\to V$ be the maps which give each edge the associated left and right vertices. Let $W_{\Gamma, \{n_e\}}(\PP_\epsilon^L)$ be the graph integral
\begin{align*}
    W_{\Gamma, \{n_e\}}(\PP_\epsilon^L)=\prod_{v\in V}\int_{\Etau} {d^2 w_v\over \Im \tau} \prod_{e\in E} \pa^{n_e}_{w_{l(e)}} \PP_\epsilon^L(w_{l(e)}, w_{r(e)};\tau, \bar\tau)
\end{align*}
where $n_e$'s are non-negative integers that associates to each $e\in E$, and $\PP_\epsilon^L$ is the regularized BCOV propagator on the elliptic curve $\Etau$. Then $\lim\limits_{\substack{\epsilon\to 0\\ L\to \infty}}W_{\Gamma, \{n_e\}}(\PP_\epsilon^L)$ exists as an almost holomorphic modular form of weight $2|E|+\sum\limits_{e\in E}n_e$.
\end{lem}

\begin{prop}
    $\F^{\Etau}_g[\infty][t^{k_1}\omega,\cdots,t^{k_n}\omega]$, which is viewed as a function on $\tau\in \mathbb H$, is an almost holomorphic modular form of weight $2g-2+2n$.
\end{prop}
\begin{proof} $\F^{\Etau}_g[\infty][t^{k_1}\omega,\cdots,t^{k_n}\omega]$ is given by Feynman graph integrals of the type in the previous Lemma \ref{almost modular forms}. Let $\Gamma$ be one of such graphs. By Proposition \ref{holomorphic derivatives}, it has precisely $n$ vertices $\fbracket{v_1,\cdots, v_n}$ . Let $E$ be the number of edges, $N$ be the total number of holomorphic derivatives appearing in the local functionals for the vertices, $L$ be the number of edges. Assume that the genus at each vertex $v_i$ is $g_i$. Since the graph is connected, we have
$$
    1-L=n-E
$$
and
$$
   g=L+\sum_{i=1}^n g_i
$$

By Proposition \ref{holomorphic derivatives}, $N$ is related to the genus via
$$
     N=\sum_{i=1}^n 2g_i
$$

We conclude that the graph integral for $\Gamma$ is an almost holomorphic modular form of weight
$$
        2E+N=2g-2+2n
$$
\end{proof}
\begin{cor}
The following limit makes sense
$$
    \lim\limits_{\bar\tau\to \infty}\F^{\Etau}_g[\infty][t^{k_1}\omega,\cdots,t^{k_n}\omega]
$$
which is a quasi-modular form of weight $2g-2+2n$.
\end{cor}

\begin{rmk}
In \cite{Si-Kevin}, we give another proof of the modularity of $\F_g^{\Etau}$ using obstruction calculus and interpret the $\bar\tau\to\infty$  from the monodromy split filtration around large complex limit. Here we use the graph integrals as an explicit realization.
\end{rmk}

\subsubsection{Cohomological localization}
Let $A$ (resp.$B$) denote the homology class of the segment $[0,1]$ (resp.$[0,\tau]$) on the elliptic curve $\Etau$. Let $\alpha_A $ (resp.$\alpha_B$) be the 1-form representing the corresponding Poincare dual. As cohomology class, we have
$$
    \empty [\omega\vdash dw]=\left[{-i\over 2\Im \tau} d\bar w\right]=\left[{i\over 2\Im \tau}(\bar\tau \alpha_A- \alpha_B)\right]
$$

Consider the isomorphism of complexes
\begin{eqnarray*}
    \Phi: \left(\PV^{1,*}_{\Etau}\oplus t \PV^{0,*}_{\Etau}, Q\right)&{\to}& \left(\mc A^{*,*}, d\right)\\
      \alpha+t\beta&\to& (\alpha+\beta)\vdash dw
\end{eqnarray*}
where $\mc A^{*,*}$ is the space of smooth differential forms on $\Etau$. Let
$$
   \omega_A=\Phi^{-1}(\alpha_A), \ \ \omega_B=\Phi^{-1}(\alpha_B)
$$

It follows that there exists $\beta\in \PV^{1,*}\oplus t \PV^{0,*}$ such that
$$
    \omega={i\over 2\Im \tau}(\bar\tau \omega_A- \omega_B)+Q\beta
$$

Let $A_1, \cdots, A_n$ (resp.$B_1,\cdots, B_n$) be disjoint cycles on $\Etau$ which lie in the same homology class of $A$ (resp.$B$). The quantum master equation at $L=\infty$ says
\begin{align*}
Q\F^{\Etau}_g[\infty]=0
\end{align*}
which implies that
$$
     \F^{\Etau}_g[\infty][t^{k_1}\omega,\cdots,t^{k_n}\omega]= \F^{\Etau}_g[\infty] [t^{k_1}{i\over 2\Im \tau}(\bar\tau \omega_{A_1}- \omega_{B_1}),\cdots,t^{k_n}{i\over 2\Im \tau}(\bar\tau \omega_{A_n}- \omega_{B_n})]
      $$
Under the limit $\bar\tau\to \infty$, we have
\begin{align}
   \lim\limits_{\bar\tau\to \infty}\F^{\Etau}_g[\infty][t^{k_1}\omega,\cdots,t^{k_n}\omega]=\lim\limits_{\bar\tau\to \infty}\F^{\Etau}_g[\infty][t^{k_1}\omega_{A_1},\cdots,t^{k_n}\omega_{A_n}]
\end{align}

Since the supports of $\omega_{A_i}$'s are disjoint, and the propagator is concentrated at $\PV_{\Etau}^{0,0}$, the RHS can be represented as Feynman graph integrals
\begin{align}\label{BCOV-Feynman-graph}
 &\lim\limits_{\bar\tau\to \infty}\sum_{g\geq 0}\hbar^{g-1}\F^{\Etau}_g[\infty][t^{k_1}\omega_{A_1},\cdots,t^{k_n}\omega_{A_n}] \nonumber\\
=&\sum\limits_{\substack{\Gamma: \text{connected graph}\\ |V(\Gamma)|=n}}W_\Gamma\left({\hbar \PP^\infty_{0}(\tau;\infty)}; {1\over \hbar}\int_{A_1}dw \mc J^{(k_1)}, \cdots, {1\over \hbar}\int_{A_n}dw \mc J^{(k_n)}\right)
\end{align}
where we sum over all connected Feynman graph integrals with $n$ vertices, with propagator $\hbar \PP_0^\infty(\tau;\infty)$, and the $i$th vertex given by ${1\over \hbar}\int_{A_i}dw \mc J^{(k_i)}$. Here $\mc J^{(k)}=\sum\limits_{g\geq 0}\hbar^g \mc J^{(k)}_g$, and $\mc J^{(k)}_g$ is a lagrangian on $\PV^{0,0}_{\Etau}$ which contains $2g$ holomorphic derivatives by Proposition \ref{holomorphic derivatives}.

We will use $\alpha$ to represent a general element in $\PV^{0,0}_{\Etau}$, and write
\begin{align}\label{alpha-definition}
    \alpha^{(n)}=\left(\sqrt{\hbar}{\pa\over \pa w}\right)^n\alpha, \ \ \alpha^{(0)}\equiv \alpha
\end{align}
$\mc J^{(k)}(\alpha)$ can be naturally viewed  as an element in $\C[\alpha,\alpha^{(1)},\cdots]/ \Im D$, where
\begin{eqnarray*}
    D=\sum_{i=0}^\infty \alpha^{(i+1)}{\pa\over \pa \alpha^{(i)}}
\end{eqnarray*}
represents the operator of total derivative. The initial condition is determined by the classical BCOV action, which says
\begin{eqnarray*}
    \mc J^{(k)}_0(\alpha)={1\over (k+1)!}\alpha^{k+2}
\end{eqnarray*}

If we assign the following degree
\begin{eqnarray*}
   \deg \alpha^{(n)}=n+1
\end{eqnarray*}
then the Hodge weight condition implies
\begin{eqnarray*}
      \deg \mc J^{(k)}=(k+2)
\end{eqnarray*}

In particular, the above degree constraint tells us that
\begin{eqnarray*}
    \mc J^{(0)}={1\over 2}\alpha^2
\end{eqnarray*}
and
\begin{eqnarray*}
   \mc J^{(1)}={1\over 3!}\alpha^3
\end{eqnarray*}
where the other possible terms don't contribute since they are in the image of $D$.

\subsubsection{Theory on $\C^*$ and the commutativity property}
We explore the properties of $\mc J^{(k)}$ by considering the BCOV theory on $\C^*$.

Since the quantization problem is local \cite{Costello-book} and $\F^{\Etau}[L]$ is constructed from the translation invariant quantization on $\C$ from local chiral functionals, the same local functionals define a scale-invariant quantization on $\C^*$. Let's call it $\F^{\C^*}[L]$.

We will use $z$ to denote the coordinate on $\C^*$
$$
   z=\exp(2\pi i w)
$$
where $w$ is the coordinate on $\C$. The holomorphic volume form is ${dz\over z}$ which defines the trace operator on polyvector fields. We will use $\OO_{\C^*}$ to denote the space of holomorphic functions on $\C^*$.  Let $C_r$ be the circle $\{z\in \C^*| |z|=r\}, r>0$. We associate $C_r$ a non-negative smooth function $\rho_{C_r}$, which takes constant value outside a small neighborhood of $C_{r}$, such that  $\rho_{C_r}=1$ when $|z|\gg r$ and $\rho_{C_r}=0$ when $|z|\ll r$. Note that $d\rho_{C_r}$ is the generator of $\HH^1_c(\C^*)$, with
$$
     \int_{\C^*} d(\rho_C)\wedge {dz\over 2\pi i z}=1
$$
and $d\rho_{C_r}$ represents the Poincare dual of $C_{r}$. Let $\omega_{C_r}$ be the following polyvector field
$$
       \omega_{C_r}=Q(\rho_{C_r} z\pa_z)\in \PV^{1,1}_{\C^*}\oplus t \PV^{0,0}_{\C^*}
$$
and consider
$$
        {\mc O}_{C_{r_1}, C_{r_2}}^{(k_1,k_2)}[L]= e^{-F[L]/\hbar}{\pa\over \pa (t^{k_1}\omega_{C_{r_1}})}{\pa\over \pa (t^{k_2}\omega_{C_{r_2}})}e^{F[L]/\hbar}
$$
where $r_1\neq r_2$ such that the supports of $d\rho_{C_{r_1}}$ and $d\rho_{C_{r_2}}$ are disjoint.
\begin{lem}
The effective action $\F^{\C^*}[L]+\delta {\mc O}_{C_{r_1}, C_{r_2}}^{(k_1,k_2)}[L]$ satisfies renormalization group flow equation and quantum master equation, where $\delta$ is an odd variable with $\delta^2=0$.
\end{lem}
\begin{proof}
\begin{eqnarray*}
     {\mc O}_{C_{r_1}, C_{r_2}}^{(k_1,k_2)}[L] e^{\F^{\C^*}[L]/\hbar}&=&{\pa\over \pa (t^{k_1}\omega_{C_{r_1}})}{\pa\over \pa (t^{k_2}\omega_{C_{r_2}})}e^{\F^{\C^*}[L]/\hbar}\\
     &=&e^{\hbar\pa_{P_\epsilon^L}}{\pa\over \pa (t^{k_1}\omega_{C_{r_1}})}{\pa\over \pa (t^{k_2}\omega_{C_{r_2}})}e^{\F^{\C^*}[\epsilon]/\hbar}\\
     &=&e^{\hbar\pa_{P_\epsilon^L}}{\mc O}_{C_{r_1}, C_{r_2}}^{(k_1,k_2)}[\epsilon] e^{\F^{\C^*}[\epsilon]/\hbar}
\end{eqnarray*}
which proves the renormalization group flow equation. Since $Q(t^{k_1}\omega_{C_{r_1}})=Q(t^{k_2}\omega_{C_{r_2}})=0$,
\begin{eqnarray*}
     \bracket{Q+\hbar\Delta_L}{\mc O}_{C_{r_1}, C_{r_2}}^{(k_1,k_2)}[L] e^{\F^{\C^*}[L]/\hbar}&=&\bracket{Q+\hbar\Delta_L}{\pa\over \pa (t^{k_1}\omega_{C_{r_1}})}{\pa\over \pa (t^{k_2}\omega_{C_{r_2}})}e^{\F^{\C^*}[L]/\hbar}\\
     &=&{\pa\over \pa (t^{k_1}\omega_{C_{r_1}})}{\pa\over \pa (t^{k_2}\omega_{C_{r_2}})}\bracket{Q+\hbar\Delta_L}e^{\F^{\C^*}[L]/\hbar}\\&=&0
\end{eqnarray*}
which proves the quantum master equation.
\end{proof}
We will consider the restriction of ${\mc O}_{C_{r_1}, C_{r_2}}^{(k_1,k_2)}[L]$ as a functional on holomorphic functions $\OO_{\C^*}\subset \PV^{0,0}_{\C^*}$ in the following discussion, and we still denoted it by ${\mc O}_{C_{r_1}, C_{r_2}}^{(k_1,k_2)}[L]$. Since $\omega_{C_{r_1}}, \omega_{C_{r_2}}$ have compact support, we can take $L\to \infty$ to obtain $\mc O_{C_{r_1}, C_{r_2}}^{(k_1,k_2)}[\infty]$ as a functional on $\OO_{\C^*}$. Note that elements in $\OO_{\C^*}$ lie in the kernel of $Q=\bar\pa-t\pa$. Quantum master equation implies that $\mc O_{C_{r_1}, C_{r_2}}^{(k_1,k_2)}[\infty]$ only depends on the homology class of $C_{r_1}, C_{r_2}$ and the integers $k_1,k_2$ if we restrict on $\OO_{\C^*}$. Following the convention as in Equation (\ref{BCOV-Feynman-graph}), we have the following
\begin{lem} Restricting on $\OO_{\C^*}$, then
\begin{align}
  {\mc O}_{C_{r_1}, C_{r_2}}^{(k_1,k_2)}[\infty]=\exp\left( \hbar \pa_{\PP_0^\infty} \right)\left({1\over \hbar}\int_{C_1}{dz\over 2\pi i z} \mc J^{(k_1)}{1\over \hbar}\int_{C_2}{dz\over 2\pi i z} \mc J^{(k_2)}\right)
\end{align}
where $\PP_\epsilon^L$ is the regularized BCOV propagator on $\C^*$
\begin{eqnarray*}
      \PP_\epsilon^L(z_1,z_2)=-{1\over \pi}\int_\epsilon^L {dt\over 4\pi t}\sum_{n\in \Z}\left({\bar w_{1}-\bar w_2+n\over 4t}\right)^2e^{-{|w_1-w_2+n|^2/4t}}
\end{eqnarray*}
and
\begin{eqnarray*}
    \PP_0^\infty(z_1,z_2)= \lim_{\substack{\epsilon\to 0\\ L\to \infty}}\PP_\epsilon^L(z_1,z_2)={1\over (2\pi i)^2}\sum_{n\in \Z}{1\over (w_1-w_2-n)^2}={z_1z_2\over (z_1-z_2)^2}
\end{eqnarray*}
here $z_k=\exp(2\pi i w_k), k=1,2$.
\end{lem}
\begin{proof}
Let $S=\lim\limits_{L\to 0}\F^{\C^*}[L]$ be the classical local functional which is chiral by construction (see Prop \ref{admissible-quantization}). Since $S$ is local and the supports of ${ t^{k_1}\omega_{C_{r_1}}}$ and $t^{k_2}\omega_{C_{r_2}}$ are disjoint, we have
\begin{align*}
             {\mc O}_{C_{r_1}, C_{r_2}}^{(k_1,k_2)}[\infty] e^{\F^{\C^*}[\infty]/\hbar}&={\pa\over \pa (t^{k_1}\omega_{C_{r_1}})}{\pa\over \pa (t^{k_2}\omega_{C_{r_2}})} e^{\hbar \pa_{\PP_0^\infty}} e^{S/\hbar}\\
             &=e^{\hbar \pa_{\PP_0^\infty}} \bracket{{1\over \hbar}{\pa\over \pa (t^{k_1}\omega_{C_{r_1}})}S } \bracket{{1\over \hbar }{\pa\over \pa (t^{k_2}\omega_{C_{r_2}})}S}e^{S/\hbar}\\
             &=e^{\hbar \pa_{\PP_0^\infty}} \bracket{{1\over \hbar}\int_{C_1}{dz\over 2\pi i z} \mc J^{(k_1)} }\bracket{{1\over \hbar}\int_{C_2}{dz\over 2\pi i z} \mc J^{(k_2)}}e^{S/\hbar}
\end{align*}
If we restrict to $\OO_{\C^*}$ as the inputs, the degree constraint will imply that
\begin{align*}
  {\mc O}_{C_{r_1}, C_{r_2}}^{(k_1,k_2)}[\infty]=e^{\hbar \pa_{\PP_0^\infty}}\left({1\over \hbar}\int_{C_1}{dz\over 2\pi i z} \mc J^{(k_1)}{1\over \hbar}\int_{C_2}{dz\over 2\pi i z} \mc J^{(k_2)}\right)
\end{align*}
\end{proof}

Since ${\mc O}_{C_{r_1}, C_{r_2}}^{(k_1,k_2)}[\infty]$ only depends on the homology class of $C_{r_1}, C_{r_2}$, we have
\begin{eqnarray*}
      \exp\left( \hbar \pa_{\PP_0^\infty} \right)\left(\int_{C_{r_1}}{dz_1\over 2\pi iz_1} \mc J^{(k_1)}\int_{C_{r_2}-C_{r_3}}{dz_2\over 2\pi i z_2} \mc J^{(k_2)}\right)=0
\end{eqnarray*}
where $0<r_3<r_1<r_2$. We call this the \emph{commutativity property}. The readers  familiar with conformal field theory will realize that this is precisely an OPE relation for operators in free chiral bosonic system. The contraction with the propagator $\PP_0^\infty$ produces a rational function with poles at the diagonal and $\int_{C_{r_2}-C_{r_3}}$ computes precisely the residue \cite{Polchinski-String}.

\begin{lem}\label{uniqueness}$\mc J^{(k)}$'s are uniquely determined (up to total derivative) by the initial conditions $\mc J^{(k)}(\alpha)={1\over (k+2)!}\alpha^{k+2}+ O(\hbar), \alpha\in \OO_{\C^*}$ and the above commutativity property.
\end{lem}
\begin{proof} If the propagator $\hbar \PP_0^\infty$ connects one $\alpha^{(n)}$ from $\mc J^{(k_1)}(\alpha)$ and one $\alpha^{(m)}$ from $\mc J^{(k_2)}(\alpha)$, we see from (\ref{alpha-definition}) that it replaces the two terms by
$$
     \hbar \left(\sqrt{\hbar}z_1\pa_{z_1}\right)^n  \left(\sqrt{\hbar}z_2\pa_{z_2}\right)^m {z_1z_2\over (z_1-z_2)^2}
     =(-1)^{m+1}\left(\sqrt{\hbar}z_1\pa_{z_1}\right)^{n+m+1}\left({\sqrt{\hbar}z_2\over z_1-z_2} \right)
$$

Using residue we see that
$$
      \exp\left( \hbar \pa_{\PP_0^\infty} \right)\left(\int_{C_1}{dz_1\over 2\pi iz_1} \mc J^{(k_1)}\int_{C_2-C_2^\prime}{dz_2\over 2\pi i z_2} \mc J^{(k_2)}\right)
$$
gives rise to a local functional $\int_{C_1}{dz\over 2\pi i z} I$ where $I(\alpha)\in \C[\alpha, \alpha^{(1)},\cdots][\sqrt{\hbar}, \sqrt{\hbar}^{-1}]$. Let
$$
      u^{(n)}(z_1,z_2)=-\left(\sqrt{\hbar}z_1\pa_{z_1} \right)^{n}\left(\sqrt{\hbar} z_2\over z_1-z_2 \right)
$$

We have
$$
   u^{(n)} u^{(m)}={n! m!\over (n+m+1)!}u^{(n+m+1)}+\sqrt{\hbar} f(u^{(k)},\sqrt{\hbar})
$$
where $f$ is a polynomial which is linear in $u^{(k)}$'s. This can be proved easily by induction on $n$.

It follows from that $I\in \sqrt{\hbar}\C[\alpha, \alpha^{(1)},\cdots][\sqrt{\hbar}]$. We consider the leading $\sqrt{\hbar}$ term in  $
      \exp\left( \hbar \pa_{P_0^\infty} \right)\left(\int_{C_1}{dz_1\over 2\pi iz_1} \mc J^{(k_1)}\int_{C_2-C_2^\prime}{dz_2\over 2\pi i z_2} \mc J^{(1)}\right)
$, where $\mc J^{(1)}(\alpha)={1\over 3!}\alpha^3$. If there's only one propagator, it replaces each term $\alpha^{(n)}$ in $\mc J^{(k_1)}$ by
$$
   {1\over 2}\int_{C_{z_1}}{dz_2\over 2\pi i z_2} \left(\sqrt{\hbar}z_1\pa_{z_1}\right)^{n+1} \left( {\sqrt\hbar z_2\over z_1-z_2} \alpha(z_2)^2 \right)
   =\sqrt{\hbar}\left(\sqrt{\hbar}z_1\pa_{z_1}\right)^{n+1}\left({1\over 2}\alpha(z_1)^2\right)
$$
where $C_{z_1}$ is a small loop around $z_1$. If there're two propagators, then it replaces each pair $\alpha^{(n)}, \alpha^{(m)}$ in $\mc J^{(k_1)}$ by
\begin{eqnarray*}
       && \int_{C_{z_1}}{dz_2\over 2\pi i z_2} \left(\sqrt{\hbar}z_1\pa_{z_1}\right)^{n+1} \left( {\sqrt\hbar z_2\over z_1-z_2} \right)\left(\sqrt{\hbar}z_1\pa_{z_1}\right)^{m+1} \left( {\sqrt\hbar z_2\over z_1-z_2} \right) \alpha(z_2)\\
       &=& \int_{C_{z_1}}{dz_2\over 2\pi i z_2} {(n+1)!(m+1)!\over (n+m+3)!}\left(\sqrt{\hbar}z_1\pa_{z_1}\right)^{n+m+3}\left( {\sqrt\hbar z_2\over z_1-z_2} \right)\alpha(z_2)+ \mbox{higher order in} \sqrt{\hbar}
\end{eqnarray*}
Therefore we find that the leading $\sqrt{\hbar}$ term in $
      \exp\left( \hbar P_0^\infty \right)\left(\int_{C_1}{dz_1\over 2\pi iz_1} \mc J^{(k_1)}\int_{C_2-C_2^\prime}{dz_2\over 2\pi i z_2} \mc J^{(1)}\right)
$ is given by
$$
  \int_{C_1}{dz_1\over 2\pi i z_1} E \mc J^{(k_1)}
$$
where $E$ is the operator
$$
   E={1\over 2}\sum_{k,l\geq 0}{(k+l)!\over k! l!}\alpha^{(k)}\alpha^{(l)} {\pa\over \pa \alpha^{(k+l-1)}} + \sum_{k,l\geq 0}{(k+1)!(l+1)!\over (k+l+3)!}\alpha^{(k+l+3)}{\pa\over \pa \alpha^{(k)}}{\pa\over \pa \alpha^{(l)}}
$$

By the commutative property, $E\mc J^{(k_1+1)}(\alpha)$ is a total derivative, i.e., lies in the image of $D$. The uniqueness now follows from the lemma below.
\end{proof}

\begin{lem}\label{existence}Consider the graded ring $A=\C[\alpha^{(0)},\alpha^{(1)},\cdots]/\im D$ with grading given by
$$
  \deg \alpha^{(k)}=k+1
$$
where $D$ is the operator of degree 1
$$
    D=\sum_{n\geq 0}\alpha^{(n+1)}{\pa\over \pa\alpha^{(n)}}
$$
Let $E$ be the operator of degree 2 acting on $A$
$$
   E={1\over 2}\sum_{k,l\geq 0}{(k+l)!\over k! l!}\alpha^{(k)}\alpha^{(l)} {\pa\over \pa \alpha^{(k+l-1)}} + \sum_{k,l\geq 0}{(k+1)!(l+1)!\over (k+l+3)!}\alpha^{(k+l+3)}{\pa\over \pa \alpha^{(k)}}{\pa\over \pa \alpha^{(l)}}
$$
There there exists unique $\mc J^{(k)}\in A$ of degree $k+1$ such that
$$
          E \mc J^{(k)}=0
$$
\end{lem}
\begin{proof}Let $E=E_1+E_2$ where
\begin{eqnarray*}
 E_1&=&{1\over 2}\sum_{k,l\geq 0}{(k+l)!\over k! l!}\alpha^{(k)}\alpha^{(l)} {\pa\over \pa \alpha^{(k+l-1)}},\\
  E_2&=& \sum_{k,l\geq 0}{(k+1)!(l+1)!\over (k+l+3)!}\alpha^{(k+l+3)}{\pa\over \pa \alpha^{(k)}}{\pa\over \pa \alpha^{(l)}}
\end{eqnarray*}

It's easy to check that
$$
      [D, E_1]=[D, E_2]=0
$$

We write $E_1=E^\prime_1+ D\alpha^{(0)}$, where $D\alpha^{(0)}$ is the operator composed of multiplication by $\alpha^{(0)}$ and $D$, and
$$
    E^\prime_1={1\over 2}\sum_{k,l> 0}{(k+l)!\over k! l!}\alpha^{(k)}\alpha^{(l)} {\pa\over \pa \alpha^{(k+l-1)}} -\alpha^{(1)}
$$

We  can choose a basis of $A=\C[\alpha^{(0)},\alpha^{(1)},\cdots]/\im D$ as
$$
      \{\alpha^{(i_1)}\alpha^{(i_2)}\cdots (\alpha^{(i_k)})^2\}, \ \ \ \ 0\leq i_1\leq i_2\leq \cdots\leq i_k
$$
$E_1^\prime$ acts on the above basis in the obvious way, while for the action of $E_2$, we need to transform the result of the action to the above basis using the operator $D$.\\ \\
{\bf Claim} $\ker E^\prime_1=Span\{\left(\alpha^{(0)}\right)^k\}_{k\geq 0}$.\\

To prove the claim, we consider the filtration by the number of $\alpha^{(1)}$
$$
       F^p A= \left(\alpha^{(1)}\right)^p A, \
$$
then
$$
      E_1^\prime\equiv \alpha^{(1)}\left(\sum_{k\geq 2}(k+1)\alpha^{(k)}{\pa\over \pa \alpha^{(k)}}+\alpha^{(1)}{\pa\over \pa \alpha^{(1)}}-1\right)
      : Gr^p_{F^\bullet} A\to Gr^{p+1}_{F^\bullet}A
$$
where $\left(\sum\limits_{k\geq 2}(k+1)\alpha^{(k)}{\pa\over \pa \alpha^{(k)}}+\alpha^{(1)}{\pa\over \pa \alpha^{(1)}}-1\right)$ is a rescaling operator on $A$, which is positive on $\alpha^{(i_1)}\alpha^{(i_2)}\cdots (\alpha^{(i_k)})^2$ if $i_k\geq 1$. For $\left(\alpha^{(0)}\right)^k$, we have
$$
     E_1^\prime \left(\alpha^{(0)}\right)^k=-\left(\alpha^{(0)}\right)^k\alpha^{(1)}=-{1\over k+1}D\left(\alpha^{(0)}\right)^{k+1}
$$
which is zero in $A$. This proves the claim. \\

Let $A^{(k)}$ be the degree $k$ part of $A$ which is finite dimensional. We consider the second homogeneous grading on $A^{(k)}$ by giving all $\alpha^{(k)}$ homogeneous degree 1. Then $E^\prime_1$ is homogeneous of degree $1$ and $E_2$ is homogeneous of degree $-1$. Let $f\in A^{(k)}$ such that $Ef=0$. We decompose
$$
   f=\sum_{i=0}^k f_i
$$
where $f_k$  is homogeneous of degree $i$. Therefore we have
\begin{eqnarray*}
       &E_1f_k=0, E_1 f_{k-1}=0&\\ &E_2 f_i=-E_1 f_{i-2}&, 2\leq i \leq k
\end{eqnarray*}

It follows from the claim that $f_k$ is a multiple of $\left(\alpha^{(0)}\right)^k$ and all the other $f_i$'s are uniquely determined. This proves the uniqueness.

 To show the existence, we consider the lagrangian in Eqn (\ref{definition-lagrangian})
\begin{eqnarray*}
&&\int_C {dz\over 2\pi i z} \sum_{k\geq -1}\lambda^{k+1} \mc L^{(k)}(\alpha(z))\\
    &=&\int_{C}{dz\over 2\pi i z} \exp\left(S(\lambda z\pa_z)(\lambda \alpha(z)) \right), \ \ S(t)={\sinh t/2\over t/2}\\
    &=&\int_{C}{dz\over 2\pi i z} \exp\left(\left(e^{\lambda z\pa_z/2}-e^{-\lambda z\pa_z/2}\right)\phi(z)) \right), \ \ \alpha(z)=z\pa_z\phi(z)\\
    &=&\int_{C}{dz\over 2\pi i z} \exp\left(\phi(e^{\lambda/2}z)-\phi(e^{-\lambda/2}z)\right)\\
    &=&\int_C {dz\over 2\pi i z} \exp\left(\phi(e^\lambda z)-\phi(z)\right)\\
    &=&\int_C {dz\over 2\pi i z} e^{-\phi(z)}e^{\lambda z\pa_z} e^{\phi(z)}\\
    &=&\int_C {dz\over 2\pi i z} \suml_{k\geq 0}{\lambda^k \over k!} \bracket{z\pa_z+\alpha}^k \cdot 1
\end{eqnarray*}
where we use the convention that $\mc L^{(-1)}=1$. Now we view $\alpha(z)$ as the bosonic field of the free boson system described in section \ref{mirror-boson-fermion}. Since the normal ordered operator
$$
   {1\over S(\lambda)}\int_{C}{dz\over 2\pi i z} :\exp\left(S(\lambda z\pa_z)(\lambda \alpha(z)) \right):_{B}
$$
is the bosonization of the fermionic operator
$$
    \int_C {dz\over 2\pi i z} b(e^{\lambda/2}z)c(e^{-\lambda/2}z)
$$
which is already simultaneously diagonalized on the standard fermionic basis. It follows that
$$
       \int_C {dz\over 2\pi i z} :\mc L^{(k)}:_B
$$
are commuting operators on the bosonic Fock space, where the normal ordering relation is given by
$$
   \alpha(z_1)\alpha(z_2)= {z_1z_2\over (z_1-z_2)^2}+:\alpha(z_1)\alpha(z_2):_B \ \ \text{if}\ |z_1|>|z_2|
$$
If we rescale $\lambda\to \sqrt{\hbar} \lambda$, then
$$
   \int_C{dz\over 2\pi i z} :{1\over (k+1)!}(\sqrt{\hbar}z\pa_z+\alpha)^{k+1} 1:_B
$$
are commutating operators on bosonic Fock space if we impose the normal ordering relation
$$
    \alpha(z_1)\alpha(z_2)={\hbar z_1z_2\over (z_1-z_2)^2}+:\alpha(z_1)\alpha(z_2):_B \ \ \text{if}\ |z_1|>|z_2|
$$
This is precisely the commutativity property, i.e., we can take
$$
   \mc J^{(k)}={1\over (k+1)!}\bracket{D+\alpha^{(0)}}^{k+1}1 \in A
$$
 This proves the existence.
\end{proof}

\subsection{Proof of mirror symmetry}\label{mirror-proof}

In this section, we prove Theorem \ref{main theorem}.

\begin{proof}[Proof of Theorem \ref{main theorem}] By Corollary \ref{hbar stationary GW}
$$
   { \sum_{d\geq 0}q^d \hbar^g\left\langle \prod_{i=1}^n\tau_{k_i}(\tilde \omega)\right\rangle_{g,n,d}}=
\lim_{\bar\tau\to \infty}\lim_{\substack{\epsilon\to 0\\ L\to \infty}}W\left(\hbar P_\epsilon^L(\tau,\bar\tau); {1\over \hbar}\int_{C_1}dw  \mathcal L^{(k_1+1)},
 \cdots, {1\over \hbar}\int_{C_n}dw  \mathcal L^{(k_n+1)}\right)
$$
where $W$ is the summation of all connected Feynman diagrams with propagator $\hbar P_\epsilon^L(\tau, \bar\tau)$ and $n$ vertices given by
$$
\int_{C_i}dw  \mathcal L^{(k_i+1)}  , \ \ 1\leq i\leq n
$$
where $C_i$'s are cycles on the elliptic curve $\Etau$ representing the class $[0,1]$ and are chosen to be disjoint, and $\mc L^{(k)}$ is the local functional on $\PV^{0,0}_{\Etau}$ defined in Eqn (\ref{definition-lagrangian}).

On the other hand, we have
$$
 \lim_{\bar\tau\to \infty} F^{\Etau}_g[\infty]\left[t^{k_1}\omega,\cdots, t^{k_n}\omega\right]
= \lim_{\bar\tau\to\infty}\lim_{\substack{\epsilon\to 0\\ L\to \infty}}W\left({\hbar P^L_{\epsilon}(\tau;\bar\tau)};  {1\over \hbar}\int_{C_1}dw  \mathcal J^{(k_1)},
 \cdots, {1\over \hbar}\int_{C_n}dw  \mathcal J^{(k_n)}\right)
$$
where $\mc J^{(k)}=\sum\limits_{g\geq 0}\hbar^g \mc J^{(k)}_g$ are local lagrangians on $\PV^{0,0}_{\Etau}$ which contain only holomorphic derivatives. By lemma \ref{uniqueness} and the proof of existence in lemma \ref{existence},
$$
        \int_{C_i} dw  \mathcal J^{(k)}= \int_{C_i} dw  \mathcal L^{(k+1)}
$$
This proves the theorem.
\end{proof}

%%%%%%%%%  Appendix  %%%%%%%%%%%%%%%
%  \appendix
%\section{Apendix A}

\bibliography{biblio}

\address{\tiny DEPARTMENT OF MATHEMATICS, NORTHWESTERN UNIVERSITY, 2033 SHERIDAN ROAD, EVANSTON IL 60208.} \\
\indent \footnotesize{\email{sili@math.northwestern.edu}}
%%%%%%%%%  References %%%%%%%%%%%%%
\iffalse

\fi

\end{document}